\documentclass{amsart}[11pt]

\usepackage{setspace}
\setdisplayskipstretch{3}

\usepackage{amsmath,amssymb,amsfonts,amsthm}
\usepackage[all]{xy}
\usepackage{color}

\usepackage{hyperref}

\usepackage{multirow}
\usepackage{rotating}
\usepackage{fullpage}
\usepackage{graphicx}
\IfFileExists{CJKutf8.sty}{\usepackage{CJKutf8}}{}
\newcommand{\pic}{{\rm Pic(Vect)}}
\newcommand{\vect}{{\rm Vect}}
\newcommand{\aut}{{\rm Aut(Vect)}}
\newcommand{\id}{{\rm id}}

\newtheorem{theorem}{Theorem}[section]
\newtheorem{proposition}[theorem]{Proposition}
\newtheorem{lemma}[theorem]{Lemma}

\theoremstyle{theorem}
\newtheorem{definition}[theorem]{Definition}

\theoremstyle{definition}
\newtheorem{example}[theorem]{Example}

\theoremstyle{remark}
\newtheorem{remark}[theorem]{Remark}

\theoremstyle{remark}
\newtheorem{conjecture}[theorem]{Conjecture}

\numberwithin{equation}{section}

\title{Boundary Conditions for Topological Quantum Field Theories, Anomalies and Projective Modular Functors}
\author{Domenico Fiorenza}
\address{Dipartimento di Matematica,``La Sapienza'' Universit\`a di Roma, P.le Aldo Moro, 5, Roma, Italy}
\author{Alessandro Valentino}
\address{Fachbereich Mathematik, Universit\"at Hamburg, Bundesstrasse 55, 20146 Hamburg, Germany}
\begin{document}

\maketitle
\begin{abstract}
We study boundary conditions for extended topological quantum field theories (TQFTs) and their relation to topological anomalies. We introduce the notion of TQFTs with moduli level $m$, and describe extended anomalous theories as natural transformations of invertible field theories of this type. We show how in such a framework anomalous theories give rise naturally to homotopy fixed points for $n$-characters on $\infty$-groups. By using dimensional reduction on manifolds with boundaries, we show how boundary conditions for $n+1$-dimensional TQFTs produce $n$-dimensional anomalous field theories. Finally, we analyse the case of fully extended TQFTs, and show that any fully extended anomalous theory produces a suitable boundary condition for the anomaly field theory.  
\end{abstract}

\begin{flushright}
\IfFileExists{CJKutf8.sty}  
 {\begin{CJK*}{UTF8}{gbsn}
天下大乱，形势大好 
\end{CJK*}
}{\emph{``Great it is the confusion under the sky, the situation is excellent.''}\\
}

\end{flushright}
\setcounter{tocdepth}{1}
\tableofcontents

\section{Introduction}
In recent years, the study of boundary conditions for topological quantum field theories (TQFTs) has attracted much interest, both in the physics and mathematics literature; see for instance \cite{ks1,fsv1,fsv2,cmr,nuiten,fss,kap,kt,wgw}, among others. Namely, given an $n$-dimensional TQFT, from the mathematical point of view it is a sensible question to ask when does such a theory produce genuine numerical invariants of an $n$-dimensional manifold with boundary, rather than vectors in a state space associated to it. This is possible if we can regard the boundary not as arising from a ``cut-and-paste'' procedure implementing locality, but rather as a ``constrained'' part of the manifold. 
In general, there will be obstructions in extending a TQFT to manifolds with boundaries: the case of Reshetikhin-Turaev and Turaev-Viro TQFTs has been recently investigated in \cite{fsv1}. Both Reshetikhin-Turaev \cite{rt} and Turaev-Viro \cite{tuvir} TQFTs are extendeded topological field theories, namely these theories assign data also to manifolds of codimension 2. In the present work, we focus our attention on TQFTs that are extended down to codimension $k$, and at the same time, most importantly, extended \emph{up to infinity} to include diffeomorphisms, and their isotopies. This is the framework pioneered in \cite{lurie}, which makes extensive use of the language of $\infty$-categories, and which we find particularly suitable for our aims. Indeed, by regarding $n$-categories as $\infty$-categories, we can introduce the notion of a \emph{TQFT with moduli level $m$}: these are topological field theories which also detect information about the homotopy type of the diffeomorphisms group of manifolds up to a certain level $m$.\\
Our main motivation to introduce and study such field theories is due to the fact that they provide a very natural and elegant description of \emph{anomalous} TQFTs. It is well known, for instance, that the Reshetikhin-Turaev construction produces from a modular tensor category $\mathcal{C}$ a TQFT which is defined on a central extension of the extended 3-dimensional cobordism category \cite{walker}: namely, it gives rise only to a \emph{projective} representation of the 2-tier extended cobordism category ${\rm Cob}_{2}(3)$ taking values in $2\mbox{-}{\rm Vect}$, and the anomaly, in this context, is represented via a 2-cocycle on the modular groupoid \cite{turaev,bakalovkirillov,au,au2}. In a more modern approach, (topological) anomalies are themselves field theories in higher dimensions, and of a special kind, namely they are \emph{invertible}; anomalous TQFTs are then realised as \emph{truncated morphisms} from the trivial theory $\underline{1}$ to the given anomaly. We refer the reader to very recent works \cite{freed1,freed2} detailing this point of view. In the present work, we realise the anomaly theory as an invertible TQFT of moduli level 1 of the \emph{same} dimension as the anomalous TQFT. Namely, taking the higher morphisms into account there is no need for the involved TQFTs to be truncations of TQFTs defined in one dimension higher; rather, truncated TQFTs are a very particular example of moduli level 1 TQFTs. This provides a unified language to describe anomalous theories extended down to codimension $k$, and their category: given an anomaly theory $W$, it is the $(\infty, k-1)$-category of natural transformations between the trivial theory and $W$. Moreover, this description allows for more general anomaly theories, as explained in the text, and it has a strong representation theoretic flavour: anomalous $n$-dimensional TQFTs extended down to codimension $k$ give rise to homotopy fixed points for $k+1$-characters, a suitable and natural generalisation of group characters to the setting of $\infty$-groups. In codimension 1, these provide projective representations of the mapping class group of $n-1$-closed manifolds.\\ \\
Anomalous TQFTs and boundary conditions are expected to intertwine in a subtle relationship. The most striking example is provided by Chern-Simons theory, which should best be regarded as a field theory living on the boundary of a 4-dimensional TQFT \cite{fhlt,walker,witten}. Similarly, the Reshetikhin-Turaev theory arising from a modular tensor category $\mathcal{C}$ is induced by a 4-dimensional Crane-Yetter theory \cite{craneyetter,walker}. By basically using a \emph{dimensional reduction} procedure, we show that from a boundary condition of an (invertible) $n+1$-dimensional theory $Z$ one can obtain an anomalous TQFT, where the anomaly is induced by $Z$ itself. One sensible question to ask concerns the converse statement, i.e. the possibility of producing a boundary condition for an $n+1$-dimensional theory from the datum of an anomalous TQFT. In general, we do not expect this to hold: indeed, an anomalous TQFT with anomaly $W$ contains too little information to determine a boundary condition $\tilde{Z}$. Neverthless, when $Z$ is a fully extended theory the situation is much more amenable to treatment: via the cobordism hypothesis \emph{for manifolds with singularities}, we show that anomalous TQFTs with anomaly given by a fully extended TQFT $Z$ do indeed produce boundary conditions for $Z$. In other words, in the fully extended situation, ``truncated morphisms'' of TQFTs are just a shadow of something richer, namely TQFTs with genuine boundary conditions. This is particularly clear thanks to the formalism used to describe anomalies, namely as morphisms of TQFTs of moduli level 1.\\
The present work is organised as follows.
\\
\\
In Section \ref{infinity} we present a very gentle introduction to the language of $\infty$-categories, in the amount necessary to allow the reader acquainted with category theory to follow the rest of the paper. We also include some results we were not able to retrieve from the literature.\\
In Section \ref{cobordism} we give some basic notions concerning cobordism categories, with emphasis on properties available once we consider extension ``up to infinity''.\\
In Section \ref{tqft} we introduce the notion of an extended TQFT with moduli level $m$, and provide some examples; we show also how we recover ordinary extended TQFTs. The fully extended case is discussed as well in this section.\\
In Section \ref{sec.anomaly}, we introduce anomalies and anomalous TQFTs via the language developed in Section \ref{tqft}. For consistency, we also discuss invertible theories, and some properties of the Picard groupoid of $n$-vector spaces.\\
In Section \ref{twochar} we take a little detour to introduce $n$-characters and their homotopy fixed points, which is a subject in its own. We present the basic definitions and results needed to provide a description of anomalous TQFTs as homotopy fixed points, and we show how anomalous $n$-dimensional TQFTs in codimension 1 give rise to projective representations of the mapping class group of closed $n-1$-dimensional manifolds, hence to projective modular functors. 
\\
In Section \ref{boundary} we finally introduce boundary condition for TQFTs, providing examples in the simplest situations, and comparisons with the existing literature when needed. \\
In Section \ref{boundanom} we show how boundary conditions for invertible TQFTs give rise to anomalous theories. Moreover, we show that in the fully extended case also the contrary holds. We conclude with some remarks on recent results on 4-dimensional field theories arising from modular tensor categories.\\ \\
Not to burden the present work with technicalities of Higher Category theory, we have in several places appealed to intuition, and hence have preferred to give ``sketches'' of definitions, rather than full blown ones. We do feel the need then to be clearer concerning which aspects of our results should be regarded as rigorously established, and which ones still require a solid foundation, or at least technical details to be filled in. In the following we try to concisely state which tools we require: most of them are contained in \cite{lurie}, which, though lacking some amount of rigor in certain points, has had a wide influence in the study of TQFTs, in particular concerning their classification. See, for instance, \cite{fhlt}.\\
First, for any nonnegative integer $n$ and any group homomorphism $G\to O(n)$ we assume there exists a symmetric monoidal $(\infty,n)$-category $\mathrm{Bord}(n)^G$ of $G$-framed cobordism. Next, for any nonnegative integer $n$, we assume there exists a notion of a symmetric monoidal $n$-category $n\text{-}\mathrm{Vect}$ of $n$-vector spaces over a field $\mathbb{K}$, which, for $n=1$, reproduces the usual monoidal category of vector spaces over $\mathbb{K}$. Moreover, we require a natural equivalence of symmetric monoidal $(\infty,n-1)$-categories $\Omega(n\text{-}\mathrm{Vect})\cong (n-1)\text{-}\mathrm{Vect}$. In the last part of the present work, we assume also the cobordism hypothesis to hold, namely that a symmetric monoidal functor $Z\colon \mathrm{Bord}(n)^G\to n\text{-}\mathrm{Vect}$ is completely determined by its value on the $G$-framed point, and that this value can be any $G$-invariant fully dualizable object of $n\text{-}\mathrm{Vect}$. Finally, we assume a robust notion of lax natural transformations between strong monoidal $\infty$-functors between symmetric $(\infty, n)$-categories.
All the other results in the article are mathematically derived by these assumptions, and so they should be considered as mathematically established as soon as one is confident in assuming that in any rigorous foundation of the theory of symmetric monoidal $(\infty,n)$-categories, all of the above assumpions will have to be true. This is a widely expected to be so in the extended TQFTs/Higher Categories communities.\\
Neverthless, for $n\leq 2$ all the constructions we present here can be entirely reformulated using the language of ordinary categories, or the well established language of 2-categories and bicategories (see, e.g., \cite{benabou}). Indeed, the reader which is uncomfortable with the theory of $\infty$-category \emph{tout court} can safely substitute $k$ and $m$ in the paper with 1, and only have to deal with bicategories for the $(n\leq2)$-version of the results presented here. In particular, the main results of this article, i.e., the construction of projective representations of the mapping class groups of manifolds from anomalous TQFTs, and that boundary conditions for extended (invertible) TQFTs do produce anomalous topological theories can be both entirely expressed within a bicategorical language.
On the other hand, we have preferred to use the language of $\infty$-category because the naturality of the ideas contained in the present work become visibly clearer. Moreover, it allows us to ``see far'' in the landscape of topological quantum field theories, and permits  indeed interesting speculations, like the conjectural relation between Reshetikhin-Turaev anomalous 3d TQFT, and the 4-category $\mathrm{Braid}^\otimes$ we present in the final part of the article. These could certainly be seen as additional motivations to pursue the consolidation of the foundation of $\infty$-category theory in all its aspects. \\ \\
{\bf Acknowledgments.} The authors would like to thank Christian Blohmann and Peter Teichner for the invitation to visit Max-Planck-Institut f\"ur Mathematik in Bonn during April 2013, where the main bulk of this work has originated. Moreover, they would like to express gratitude to Joost Nuijten and Urs Schreiber for repeated interesting discussions, and to the referee for very useful comments on a first draft of this article. AV would like to thank Alexander Barvels, Nicolai Reshetikhin, Christoph Schweigert, Kevin Walker, and Christoph Wockel for useful discussions and suggestions. The work of AV is partly supported by the Collaborative Research Center 676 ``Particle, Strings and the Early Universe''. DF would like to thank the organisers of GAP XI - Pittsburgh, and Stephan Stolz for useful discussions and suggestions. 
\section{Preliminary notions on higher category theory}\label{infinity}
In this section we will collect relevant results concerning higher category theory, and in particular $\infty$-categories, which we will use in the paper, mainly following \cite{lurie,baezdolan}, to which we direct the reader for details. The experienced reader, instead, can skip this section altogether.\\
An $n$-category can be informally thought of as a mathematical structure generalizing the notion of a category: we not only have objects and morphisms, but also morphisms between morphisms, morphisms between morphisms between morphisms, and so on, up to $n$. In the case $n=2$, a precise definition can be given (see, e.g.,\cite{benabou,schom}), where the crucial difference arises between \emph{strict} and \emph{weak} $2$-category. Once we notice that a strict 2-category is equivalent to a category enriched in ${\rm Cat}$, we can give a recursive definition for strict $n$-categories as follows:
for $n\geq{2}$, a strict $n$-category is a category enriched in ${\rm Cat}^{n-1}$, the category of strict $n-1$-categories.
 The problem arises when we try to extend the above definition to obtain weak $n$-categories, i.e. an $n$-category where associativity for $k$-morphisms, etc. is only preserved up to $k+1$-morphisms, for $1\leq k\leq{n}$, which obey the necessary coherence diagrams. A rigorous definition of weak $n$-category can nevertheless be given, and there are even different equivalent ways of formalizing this notion. Basic references are \cite{barwick,barwick-schommer-pries}. It goes without saying that weak $n$-categories are those of relevance in the mathematical world.
 \begin{example}
An important example of (weak) $n$-category is that of $n$-vector spaces over a fixed characteristic 0 base field $\mathbb{K}$. For $n=0$, the $0$-category (i.e., the set) $0$-Vect is the field $\mathbb{K}$; for $n=1$ the $1$-category (i.e., the ordinary category) $1$-Vect is the category of (finite dimensional) vector spaces over $\mathbb{K}$. For $n=2$, the $2$-category $2$-Vect comes in various flavours: by 2-Vect one can mean the 2-category of Kapranov-Voevodsky 2-vector spaces \cite{kv}, or the 2-category of (finite) $\mathbb{K}$-linear categories with linear functors as morphisms and $\mathbb{K}$-linear natural transformations as 2-morphisms, or the 2-category of (finite dimensional) $\mathbb{K}$-algebras (to be thought as placeholders for their categories of right modules), with (finite dimensional) bimodules as 1-morphisms and morphisms of bimodules as 2-morphisms, as in \cite{schreiber-aqft}.\footnote{The 2-category of Kapranov-Voevodsky 2-vector spaces can be seen as the full subcategory of the 2-category of $\mathbb{K}$-algebras and bimodules on the $\mathbb{K}$-algebras of the form $\mathbb{K}^{\oplus m}$, for $m\in \mathbb{N}$.} This latter incarnation of $2$-Vect suggests an iterative definition of $n$-Vect, see \cite{fhlt}. For instance one can define $3$-Vect as the 3-category whose objects are tensor categories over $\mathbb{K}$, whose morphisms are bimodule categories, and so on. In any of its incarnations, $n$-Vect is 
an example of symmetric monoidal $n$-category. For instance, for $n=2$ the symmetric monoidal structure on the 2-category of finite $\mathbb{K}$-linear categories is induced by Deligne's tensor product \cite{deligne}.
 \end{example}
When one has $k$-morphisms for any $k$ up to infinity, one speaks of an $\infty$-category. Just to settle the notation, we give the following
 \begin{definition}
 For $n\geq{0}$, a $(\infty,n)$-category is a $\infty$-category in which every $k$-morphisms is invertible for $k>n$.
 \end{definition}
 Details, and a rigorous definition of an $(\infty,n)$-category as an $n$-fold complete Segal space can be found in \cite{barwick}; see also \cite{barwick-schommer-pries,lurie,scheimbauer}.
 Notice that in the ``definition'' above, invertibily of $k$-morphisms must be understood recursively in the higher categorical sense, i.e. up to invertible $k+1$-morphisms. In particular, any $n$-category can be extended to an $n$-\emph{discrete} $(\infty,n)$-category, i.e., an $(\infty,n)$-category in which all $k$-morphisms for $k>n$ are identities. We will often pass tacitily from $n$-categories to $n$-discrete $\infty$ categories in what follows. Moreover, given an $(\infty, n)$-category and objects $x,y\in\mathcal{C}$, there is a $(\infty,n-1)$-category ${\rm Mor}_{\mathcal{C}}(x,y)$ of 1-morphisms.
\begin{example}\label{pathgrp} The prototypical example of $\infty$-category arises from homotopy theory.
Indeed, let $X$ be a topological space. Then there is an $\infty$-category $\pi_{\leq{\infty}}(X)$, with objects given by the points of $X$, 1-morphisms given by continuous paths in $X$, 2-morphisms given by homotopies of paths with fixed end-points, 3-morphisms given by homotopies between homotopies, and so on. Since the composition of paths is only associative up to homotopy, i.e. up to a 2-morphism, $\pi_{\leq{\infty}}(X)$ is necessarily a weak $\infty$-category. Neverthless, the 2-morphism above, which is part of the data, is invertible up to 3-morphisms. Indeed, all $k$-morphisms in $\pi_{\leq{\infty}}(X)$ are invertible, hence it is a $(\infty, 0)$-category, which is usually called a $\infty$-groupoid. The guiding principle behind $\infty$-categories is that also the converse should be true, i.e. any $\infty$-groupoid arises as $\pi_{\leq{\infty}}(X)$ for some topological space, hence the theory of $(\infty, 0)$-categories can be defined via homotopy theory.
\end{example}
\begin{example}
A genuine example of an $(\infty,n)$-category with $n> 0$ is given by ${\rm Bord}(n)$, the $\infty$-category of cobordisms, which can be informally described as consisting of having points as objects, 1-dimensional bordisms as 1-morphisms, 2-dimensional bordisms between bordisms as 2-morphisms, and so on until we arrive at $n$-dimensional bordisms as $n$-morphisms, from where higher morphisms are given by diffeomorphisms and isotopies: more precisely, the $(n+1)$-morphisms are diffeomorphisms which fix the boundaries, $(n+2)$-morphisms are isotopies of diffeomorphisms, $(n+3)$-morphisms are isotopies of isotopies, and so on. This is an example of a $(\infty,n)$-symmetric monoidal category, see \cite{lurie}. A rigorous and detailed construction of ${\rm Bord}(n)$ as an $(\infty,n)$-symmetric monoidal category can be found in \cite{scheimbauer}.
\end{example}
\begin{remark} The $(\infty,n)$-category ${\rm Bord}(n)$ comes also in other ``flavours'', depending on the additional structures we equip the manifolds with: for instance orientation and $n$-framing give $(\infty,n$)-categories ${\rm Bord}(n)^{or}$ and ${\rm Bord}(n)^{fr}$, respectively. More precisely, let $G\to GL(n;\mathbb{R})$ be a group homomorphism. For any $k\leq n$, a $k$-manifold $M$ is naturally equipped with the $GL(n;\mathbb{R})$-bundle $TM\oplus \mathbb{R}^{n-k}$, and a $G$-framing for $M$ is the datum of a reduction of the structure group of $TM\oplus \mathbb{R}^{n-k}$ from $GL(n;\mathbb{R})$ to $G$. Just as in the non-framed case, $G$-framed $k$-manifolds with $k\leq n$ are the $k$-morphisms for a symmetric monoidal $(\infty,n)$-category ${\rm Bord}(n)^{G}$, called the $(\infty,n)$-category of $G$-cobordism. Notice that one can consider an equivalent category of $G$-cobordisms, where our manifolds are equipped with a $O(n)$-structure on the stable tangent bundle, and its $G$-reductions. The equivalence comes from the fact that $O(n)$ is a retract of $GL(n;\mathbb{R})$. We will implicitly make this identification later on.\\
In particular, when $G$ is the trivial group, one writes ${\rm Bord}(n)^{fr}$ for ${\rm Bord}(n)^{\{e\}}$, and calls it the $(\infty,n)$-category of framed cobordism, while when $G$ is $SO(n)$ one writes ${\rm Bord}(n)^{or}$ for ${\rm Bord}(n)^{SO(n)}$, and calls it the $(\infty,n)$-category of oriented cobordism. The unoriented case ${\rm Bord}(n)$ is obtained when $G$ is $O(n)$.
We will use ${\rm Bord}(n)$ generically to indicate one of these $G$-framed versions, unless explicitly specified.
\end{remark}
\medskip

As for any mathematical structure, there is a notion of morphisms between $\infty$-category, which are given by $\infty$-functors. Informally speaking,  an $\infty$--functor $F$ between two
$\infty$--categories $\mathcal{C}$ and $\mathcal{D}$ is a rule assigning to each $k$-morphism in $\mathcal{C}$ a $k$-morphism in $\mathcal{D}$ in a way respecting sources, targets and (higher) compositions. For instance, if one adopts the simplicial model for
$(\infty, 1)$-categories, i.e., if one looks at $(\infty, 1)$-categories as simplicial sets with internal horn-filling conditions (with $k$-morphisms corresponding to $k$-simplices), then an $\infty$-functor between $(\infty, 1)$-categories is precisely a morphism of simplicial sets. See \cite[Chapter 1]{lurie2} and \cite{lurie} for details. In particular, given two $\infty$-categories $\mathcal{C}$ and $\mathcal{D}$, we have an $\infty$-category ${\rm Fun}(\mathcal{C},\mathcal{D})$. It is immediate to see that, if $\mathcal{D}$ is $n$-discrete, then also ${\rm Fun}(\mathcal{C},\mathcal{D})$ is $n$-discrete (or, more precisely, it is equivalent to an $n$-discrete $\infty$-category). \\ \\
Given an $(\infty,n)$-category $\mathcal{C}$ we can obtain an ordinary category $\pi_{\leq 1}\mathcal{C}$, called the \emph{homotopy category} of $\mathcal{C}$, with objects given by the objects of $\mathcal{C}$, and morphisms given by 
equivalence classes of 1-morphisms up to invertible 2-morphisms
 in $\mathcal{C}$, where invertibility is understood in the $\infty$ setting. Similarly, for $k\geq{2}$ we can associate to $\mathcal{C}$ a $k$-category $\pi_{\leq k}\mathcal{C}$, called the \emph{homotopy $k$-category} of $\mathcal{C}$, with objects and morphisms up to $k-1$-morphisms given by those of $\mathcal{C}$, and $k$-morphisms given by equivalence classes of $k$-morphisms up to invertible $k+1$-morphisms. By the usual identification of $k$-categories with $k$-discrete $\infty$-categories, we have then the following 
\begin{lemma}\label{hcat}
The formation of the \emph{homotopy $n$-category} is the adjoint $\infty$-functor to the inclusion of $n$-discrete categories into $(\infty,n)$-categories, i.e., if  
$\mathcal{C}$ and $\mathcal{D}$ are $(\infty,n)$-categories, with $\mathcal{D}$ discrete, then one has a natural equivalence of $\infty$-categories
\begin{equation}
\mathrm{Fun}(\mathcal{C},\mathcal{D})\cong \mathrm{Fun}(\pi_{\leq n}\mathcal{C}, \mathcal{D}).
\end{equation}
\end{lemma}
In more colloquial terms, this is just the statement that if $\mathcal{D}$ is $n$-discrete then an $\infty$-functor $\mathcal{C}\to \mathcal{D}$ naturally factors as $\mathcal{C}\to \pi_{\leq n}\mathcal{C}\to \mathcal{D}$.\\
For any $(\infty, n)$-category $\mathcal{C}$ and an object $x\in\mathcal{C}$, we have that ${\rm End}_{\mathcal{C}}(x)={\rm Hom}_{\mathcal{C}}(x,x)$ is a monoidal $(\infty, n-1)$-category. In particular, to a monoidal $(\infty, n)$-category $\mathcal{C}$ we can canonically assign a monoidal $(\infty, n-1)$-category $\Omega\mathcal{C}:={\rm End}_{\mathcal{C}}(1_{\mathcal{C}})$, where $1_{\mathcal{C}}$ denotes the monoidal unit of $\mathcal{C}$. We will refer to $\Omega\mathcal{C}$ as the (based) loop space of $\mathcal{C}$. It can be seen as the homotopy pullback
\begin{equation}
\xymatrix{\Omega\mathcal{C}\ar@{}[dr]|(.2){\lrcorner}
\ar[r]\ar[d]&\mathbf{1}\ar[d]\\
\mathbf{1}\ar[r]&\mathcal{C}
}
\end{equation}
where $\mathbf{1}$ is the trivial monoidal category, and $\mathbf{1}\to \mathcal{C}$ is the unique monoidal functor from $\mathbf{1}$ to $\mathcal{C}$.
We can reiterate the construction to obtain a monoidal $(\infty, n-k)$-category, which we denote with $\Omega^{k}\mathcal{C}$. If $\mathcal{C}$ is also symmetric, then $\Omega^{k}\mathcal{C}$ is symmetric as well. We will denote with ${\rm Fun}^{\otimes}(\mathcal{C},\mathcal{D})$ the $(\infty, n)$-category of monoidal $\infty$-functors between $\mathcal{C}$ and $\mathcal{D}$. Any monoidal $\infty$-functor $F$ from $\mathcal{C}$ to $\mathcal{D}$ induces a monoidal $\infty$-functor $\Omega^{k}F$ from $\Omega^{k}\mathcal{C}$ to $\Omega^{k}\mathcal{D}$.
\begin{example}One has $\Omega(n\text{-}\mathrm{Vect})\simeq (n-1)\text{-}\mathrm{Vect}$ for any $n\geq 1$. For instance, the monoidal unit of the category 1-Vect is the field $\mathbb{K}$ seen as a vector space over itself, hence 
\begin{equation}
\Omega(1\text{-}\mathrm{Vect})=\mathrm{End}_{1\text{-}\mathrm{Vect}}(\mathbb{K})=\mathbb{K}=0\text{-}\mathrm{Vect}.
\end{equation}
Similarly, the monoidal unit of the 2-category 2-Vect is the category Vect, while its category of endomorphisms is the category of linear functors from Vect to Vect, which can be canonically identified with Vect itself.
\end{example} 
\begin{lemma}\label{endo}
Let $\mathcal{C}$ be a symmetric monoidal $(\infty, n)$-category, and let $\mathcal{D}$ be a symmetric monoidal $(\infty, n+1)$-category. Then
\begin{equation}
{\rm End}_{{\rm Fun}^{\otimes}(\mathcal{C},\mathcal{D})}({\underline{1}}_{\mathcal{D}})\simeq{\rm Fun}^{\otimes}(\mathcal{C},\Omega\mathcal{D})
\end{equation}
where $\underline{1}_{\mathcal{D}}\colon \mathcal{C}\to \mathcal{D}$ denotes the trivial monoidal functor, mapping all objects of $\mathcal{C}$ to the monoidal unit $1_{\mathcal{D}}$ of $\mathcal{D}$, and all morphisms in $\mathcal{C}$ to identities.
\end{lemma}
\begin{proof}
The trivial monoidal functor $\underline{1}_{\mathcal{D}}$ is the composition $\mathcal{C}\to \mathbf{1}\to \mathcal{D}$. It follows from this description that ${\rm End}_{{\rm Fun}^{\otimes}(\mathcal{C},\mathcal{D})}({\underline{1}}_{\mathcal{D}})$ is the $\infty$-category of homotopy commutative diagrams
\begin{equation}
\xymatrix{\mathcal{C}
\ar[r]\ar[d]&\mathbf{1}\ar[d]\\
\mathbf{1}\ar[r]&\mathcal{D}
\ar@{=>}(8.25,-5.25);(4.75,-8.75)
}
\end{equation}
By the universal property of the homotopy pullback, this is equivalent to ${\rm Fun}^{\otimes}(\mathcal{C},\Omega\mathcal{D})$.
\end{proof}
On the other hand, given a monoidal $(\infty, n)$-category $\mathcal{C}$ we can obtain an $(\infty, n+1)$-category $B\mathcal{C}$ with a single object, and $\mathcal{C}$ as the $\infty$-category of morphisms. We will refer to $B\mathcal{C}$ as the classifying space of $\mathcal{C}$. The relationship between $B$ and $\Omega$ is given by the following
\begin{lemma}\label{dlemma}
Let $\mathcal{C}$ be a symmetric monoidal $(\infty, n)$-category, and let $\mathcal{D}$ be a symmetric monoidal $(\infty, n+1)$-category. Then
\begin{equation}
{\rm Fun}^{\otimes}(B\mathcal{C},\mathcal{D})\simeq{\rm Fun}^{\otimes}(\mathcal{C},\Omega\mathcal{D})
\end{equation}
\end{lemma}
\begin{proof}
Let $F\in {\rm Fun}^{\otimes}(B\mathcal{C},\mathcal{D})$. Since $B\mathcal{C}$ is an $\infty$-category with a single object $\star$, and $F$ is a monoidal functor, then necessarily $F(\star)={1_{\mathcal{D}}}$. Hence, to any $k$-morphism in $B\mathcal{C}$, corresponding to a $(k-1)$-morphism in $\mathcal{C}$, is assigned by $F$ a $k$-morphism from $1_{\mathcal{D}}$ to $1_{\mathcal{D}}$ in $\mathcal{D}$, i.e., a $(k-1)$-morphism in the symmetric monoidal $(\infty,n)$-category $\Omega\mathcal{D}={\rm End}_{\mathcal{D}}({1_{\mathcal{D}}})$.
\end{proof} 
\section{Cobordism $(\infty,k)$-categories}\label{cobordism}
In this section we will recall some basic properties concerning $\infty$-categories of cobordisms. We will mainly refer to oriented cobordisms, unless otherwise stated.\\
Via the mapping cylinder construction, we obtain a monoidal embedding
\begin{equation}\label{emb}
i:{\rm Bord}(n)\hookrightarrow {\rm Bord}(n+1)
\end{equation}
Let us briefly recall how this works. 
Given a (orientation preserving) diffeomorphism $f:\Sigma_1\to \Sigma_2$ between closed $n$-dimensional oriented manifolds, the mapping cylinder of $f$ is the oriented manifold $M_{f}$ with boundary obtained as
\begin{equation}
M_{f}:=(([0,1]\times{\Sigma_1})\amalg{\Sigma_2})/\sim
\end{equation}
where $\sim$ is the equivalence relation generated by $(0,x)\sim{f(x)},\forall x\in{\Sigma_1}$.
In particular, we have that $\partial{M}_{f}=\Sigma_1\amalg{\overline{\Sigma}_2}$, where $\overline{\Sigma}_2$ denotes the manifold $\Sigma_2$ endowed with the opposite orientation, so that $M_{f}$ represents a (oriented) cobordism between $\Sigma_1$ and $\Sigma_2$. This means that $f\mapsto M_f$ maps an $(n+1)$-morphism in $\mathrm{Bord}(n)$ to an $(n+1)$-morphism in $\mathrm{Bord}(n+1)$. Moreover, the mapping cylinder construction is compatible with composition of diffeomorphisms in the following sense: 
if $g:\Sigma_1\to{\Sigma_2}$ and $f:\Sigma_2\to{\Sigma_3}$ are diffeomorphisms between closed oriented $n$-dimensional manifolds, then we have a canonical diffeomorphism
\begin{equation}
M_{fg}\simeq M_{f}\circ M_{g}\,.
\end{equation}
In other words, $f\mapsto M_f$ behaves functorially with respect to the composition of $(n+1)$-morphism. Moreover,
the mapping cylinder is compatible with isotopies of diffeomorphisms. 
Namely, an isotopy $h$ between orientation preserving diffeomorphisms $f,g\colon \Sigma_1\to \Sigma_2$ 
induces a orientation preserving diffeomorphism
\begin{equation}
h_{*}:M_{f}\xrightarrow{\simeq}{M_{g}}\,.
\end{equation}
Hence the mapping cylinder construction maps an $(n+2)$-morphism in $\mathrm{Bord}(n)$ to an $(n+2)$-morphism in $\mathrm{Bord}(n+1)$, and also in this case one can verify the compatibility with composition.
Similarly, isotopies between isotopies of diffeomorphisms produce correspondent isotopies of diffeomorphisms of the mapping cylinders.  One has natural generalisations to unoriented and to $G$-framed cobordism, and so on, so that the mapping cylinder construction actually gives an $\infty$-functor $\mathrm{Bord}(n)\to \mathrm{Bord}(n+1)$, which is immediately seen to be compatible with disjoint unions, i.e., with the monoidal structure of cobordism categories. Details on the properties of the functor $i$ can be found in \cite{lurie}: interestingly, the proof of the fact that $i$ is actually a (not full) embedding of $\infty$-categories is at the core of the Cobordism Hypothesis.
\begin{remark}
One has natural generalisations of (\ref{emb}) to unoriented, and to $G$-framed cobordisms. 
\end{remark}

Applying the iterated loop space construction to the symmetric monoidal $(\infty,n)$-category $\mathrm{Bord}(n)$ we obtain the following important
\begin{definition}
For any $0\leq k\leq n$, the $(\infty,k)$-symmetric monoidal category $\mathrm{Cob}_k^\infty(n)$ is defined as
\begin{equation}
\mathrm{Cob}_k^\infty(n):=\Omega^{n-k}\mathrm{Bord}(n)
\end{equation}
It will be called the $(\infty,k)$-category of $n$-dimensional cobordism extended down to codimension $k$.
\end{definition}
In a similar way, one can define $G$-framed cobordism categories $\mathrm{Cob}_k^{\infty,G}(n)$.\\
Note that $\mathrm{Bord}(n)=\mathrm{Cob}_n^\infty(n)$, the $(\infty,n)$-category of $n$-dimensional cobordism extended down to codimension $n$. We will refer to $\mathrm{Bord}(n)$ as the \emph{fully extended} $n$-dimensional cobordism category.\\ 
Notice that if $F\colon \mathcal{C}\to \mathcal{D}$ is a monoidal functor, then also $\Omega(F)\colon \Omega\mathcal{C}\to \Omega\mathcal{D}$ is monoidal. This in particular implies that the monoidal embedding  $i:{\rm Bord}(n)\hookrightarrow {\rm Bord}(n+1)$ induces monoidal embeddings
\begin{equation}
\mathrm{Cob}_k^\infty(n) \hookrightarrow \mathrm{Cob}_{k+1}^\infty(n+1) 
\end{equation}
for any $k\geq 0$.

\begin{remark}\label{classical-cobordism}
The homotopy category $\pi_{\leq 1}\mathrm{Cob}_1^\infty(n)$ is the usual category of $n$-dimensional cobordism: it has $(n-1)$-closed manifolds as objects and diffeomorphism classes of $n$-dimensional cobordisms as morphisms. In the following, we will refer to this category simply as $\mathrm{Cob}(n)$
\end{remark}
\begin{remark}\label{gammaenne}
The $(\infty,0)$-category $\mathrm{Cob}_0^\infty(n)$ is the $\infty$-groupoid having closed $n$-manifolds as objects, diffeomorphisms between them as 1-morphisms, isotopies between diffeomorphisms as 2-morphisms and so on.\\ 
Let $\Sigma$ be a closed $n$-dimensional manifold. By slight abuse of notation, we will denote by $B\Gamma^\infty(\Sigma)$ the connected component of $\Sigma$ in $\mathrm{Cob}_0^\infty(n)$. The homotopy category $\pi_{\leq 1}\mathrm{Cob}_0^\infty(n)$ is the groupoid usually denoted $\Gamma_n$, see \cite{bakalovkirillov}, while $\pi_{\leq 1}B\Gamma^\infty(\Sigma)$ is the (one-object groupoid associated with the) mapping class group $\Gamma(\Sigma)$ of $\Sigma$. To emphasise the $G$-framing, we will occasionally write $\Gamma^G(\Sigma)$ for the mapping class group of a $G$-framed manifold $\Sigma$. For instance, if $\Sigma$ is a closed oriented surface, then $\Gamma^{SO(2)}(\Sigma)$ is the mapping class group of isotopy classes of oriented diffeomorphisms one encounters in Teichm\"uller theory. If $\Sigma$ is a closed oriented surface endowed with a  spin structure, i.e., with a lift of the structure group $SO(2)$ of the tangent bundle to the double cover $SO(2)\xrightarrow{2:1} SO(2)$, then $\Gamma^{\mathrm{Spin}}(\Sigma)$ is the spin-framed mapping class group of $\Sigma$ considered in \cite{kriz}.  
\end{remark}
\section{Topological Quantum Field Theories}\label{tqft}
In this section we introduce the notion of a topological quantum field theory with moduli level $m$.
\subsection{TQFTs with moduli level} Since both $\mathrm{Cob}_k^\infty(n)$ and $r$-Vect are symmetric monoidal $\infty$-categories, it is meaningful to consider symmetric monoidal functors between them. This leads us to the main definition in the present work
\begin{definition}\label{modulitqft}
An $n$-dimensional TQFT extended down to codimension $k$ with moduli level $m$ is a symmetric monoidal functor
\begin{equation}
Z\colon \mathrm{Cob}_k^\infty(n)\to (m+k)\text{-}\mathrm{Vect}.
\end{equation}
\end{definition}
 \begin{remark}\label{trunc}
 One main feature of $r$-Vect, whichever realisation of $r$-vector spaces one considers, is that $\Omega(r\text{-}\mathrm{Vect})\cong (r-1)\text{-}\mathrm{Vect}$. This, together with the equivalence $\Omega\mathrm{Cob}_k^\infty(n)\cong  \mathrm{Cob}_{k-1}^\infty(n)$, implies that by looping an $n$-dimensional TQFT extended down to codimension $k$  we obtain an $n$-dimensional TQFT extended down to codimension $k-1$ with the same moduli level:
 \begin{equation}
 \Omega Z\colon \mathrm{Cob}_{k-1}^\infty(n)\to (m+k-1)\text{-}\mathrm{Vect}.
\end{equation}
On the other hand, pulling back along the inclusion $\mathrm{Cob}_{k-1}^\infty(n-1) \hookrightarrow \mathrm{Cob}_{k}^\infty(n)$ one can restrict an $n$-dimensional TQFT extended down to codimension $k$ with moduli level $m$ to a $(n-1)$-dimensional TQFT extended down to codimension $k-1$ with moduli level $m+1$,
\begin{equation}
Z\bigr\vert_{k-1}\colon \mathrm{Cob}_{k-1}^\infty(n-1)\to (m+k)\text{-}\mathrm{Vect}.
\end{equation}
We will refer to $Z\bigr\vert_{k-1}$ as the $(n-1)$-dimensional \emph{truncation} of $Z$.
\end{remark}
The terminology used in Definition \ref{modulitqft} is due to the fact that a TQFT of moduli level greater than $0$ produces in general more refined manifold invariants than an ordinary TQFT, namely it can detect the moduli space of diffeomorphisms. As we will illustrate in the following examples, from a TQFT of moduli level $k$ we can obtain in specific situations the notion of ordinary and extended TQFTs.   

\begin{example}
An $n$-dimensional TQFT extended down to codimension $1$ with moduli level $0$ is a TQFT in the sense of Atiyah and Segal \cite{atiyah,segal}. Namely, since 1-Vect is 1-discrete, a symmetric monoidal functor $Z\colon \mathrm{Cob}_{1}^\infty(n)\to 1\text{-}\mathrm{Vect}$ factors through the category $\mathrm{Cob}(n)$ of $n$-dimensional cobordism $\pi_{\leq 1}\mathrm{Cob}_{1}^\infty(n)$; see Remark \ref{classical-cobordism}. It is interesting to notice that, even if one does not a priori imposes any finite dimensionality condition on the objects in 1-Vect, i.e., if one takes 1-Vect to be the category of all vector spaces over some fixed field $\mathbb{K}$, then, as an almost immediate corollary of the definition, the vector space $Z(M)$ that an Atiyah $n$-dimensional TQFT assigns to a closed $n-1$-manifold $M$ must be finite dimensional, see \cite{bakalovkirillov,kock}.
\end{example}
\begin{example}
Similarly, an $n$-dimensional TQFT extended down to codimension $2$ with moduli level $0$ is equivalently a symmetric monoidal 2-functor
\begin{equation}
Z:{\rm Cob}_{2}(n)\to{2\mbox{-}\rm Vect}
\end{equation}
where ${\rm Cob}_{2}(n)=\pi_{\leq 2}{\rm Cob}^{\infty}_{2}(n)$ is the so-called 2-category of extended cobordism. Its objects are $(n-2)$-dimensional closed manifolds, its 1-morphisms are $(n-1)$-dimensional cobordisms, and its 2-morphisms are diffeomorphism classes of $n$-dimensional cobordisms. Such a monoidal functor is sometimes called a (2-tier) extended $n$-dimensional TQFT, see \cite{ft,morton}. Notice that applying the loop construction to an extended TQFT one obtains an $n$-dimensional TQFT in the sense of Atiyah and Segal.
\end{example}
\begin{remark}
2-tier extended TQFTs have been the subject of great investigation, in particular in 3-dimension. Indeed, historically it was 3-dimensional Chern-Simons theory which motivated the notion of an extended field theory. Particularly relevant are the extended 3d TQFTs known as of \emph{Reshetikhin-Turaev} type \cite{rt} obtained by the algebraic data encoded in a modular tensor category, and those of \emph{Turaev-Viro} type \cite{tuvir}, which are constructed from a spherical fusion category\footnote{In general, the Turaev-Viro construction produces oriented theories, while Reshetikhin-Turaev theories require a \emph{framing} to be defined.}.
\end{remark}
\begin{example}
The \emph{categorified field theories} in \cite{dsps} are an example of topological quantum field theories extended down to codimension 2 with moduli level 1.
\end{example}

\subsection{Fully extended TQFTs} It is easy to see that a $1$-dimensional TQFT in the sense of Atiyah and Segal \cite{atiyah,segal} is completely determined by the vector space $V^{+}$ it assignes to the oriented point ${\rm pt}^{+}$. Moreover, the category of 1-dimensional Atiyah-Segal TQFTs, i.e. the category
\begin{equation}
{\rm Fun}^{\otimes}(\mathrm{Cob}_1^\infty(1),1\text{-}\mathrm{Vect})
\end{equation}
turns out to be equivalent to the groupoid obtained from the category of finite dimensional vector spaces by discarding all the noninvertible morphisms. This can be seen as follows. Given a monoidal natural transformation $\varphi\colon Z_1 \to Z_2$ between two 1-dimensional Atiyah-Segal TQFTs, then we have a linear morphism $\varphi({\rm pt}^+):V_1^+ \to V_2^+$. The compatibility of $\varphi$ with the evaluation and coevaluation morphisms forces $V_1^+$ and $V_2^+$ to have the same dimension, and $\varphi({\rm pt}^+)$ to be a linear isomorphism. By the same argument one can show that $n$-dimensional Atiyah-Segal TQFTs as well form a groupoid. See \cite{freed3} for details. 
\medskip

The rigidity of the 1-dimensional example illustrated above comes from the fact that the involved TQFT is a moduli level 0 \emph{fully extended} TQFT. 
%
%
Indeed, these TQFTs encode so much information that they can be completely classified. This is indeed the content of the \emph{cobordism hypothesis}, which can be stated as follows.\footnote{Here we are formulating the cobordism hypothesis for TQFTs with target higher vector spaces; one can give a more general formulation with target an arbitrary $(\infty,n)$-symmetric monoidal category, see \cite{lurie}.} 
\begin{theorem}[Lurie-Hopkins]\label{luriehopkins}
A moduli level 0 fully extended $n$-dimensional framed TQFT is completely determined by a fully dualizable $n$-vector space. More precisely, 
let $(n\text{-}\mathrm{Vect})_{\mathrm{fd}}$ be the 
 the full subcategory 
 of $n\text{-}\mathrm{Vect}$
 of\emph{fully dualizable} objects, and let $(n\text{-}\mathrm{Vect})_{\mathrm{fd}}^{(\infty,0)}$ be the underlying $(\infty,0)$-groupoid, i.e., the $(\infty,0)$-groupoid obtained from $(n\text{-}\mathrm{Vect})_{\mathrm{fd}}$ by discarding all the non-invertible morphisms. Then there is an equivalence of $\infty$-categories
\begin{equation}\label{cobframed}
{\rm Fun}^{\otimes}({\rm Bord}^{fr}(n),n\text{-}\mathrm{Vect})\simeq(n\text{-}\mathrm{Vect})_{\mathrm{fd}}^{(\infty,0)}
\end{equation}
induced by the evaluation functor $Z\mapsto Z({\rm pt}^{+})$. More generally, if $G\to O(n)$ is a reduction of structure group for $n$-dimensional manifolds, then there is a natural action of $G$ on $(n\text{-}\mathrm{Vect})_{\mathrm{fd}}$ and $Z\mapsto Z({\rm pt}^{+})$ induces an equivalence 
\begin{equation}\label{Gcob}
{\rm Fun}^{\otimes}({\rm Bord}^{G}(n),n\text{-}\mathrm{Vect})\simeq{(n\text{-}\mathrm{Vect})^{G\, (\infty,0)}_{\mathrm{fd}}}
\end{equation}
where ${(n\text{-}\mathrm{Vect})^G_{\mathrm{fd}}}$ denotes the full subcategory on the homotopy fixed points for the induced $G$-action on $(n\text{-}\mathrm{Vect})_{\mathrm{fd}}$.
\end{theorem}

\begin{remark}
The $G$-action on $(n\text{-}\mathrm{Vect})_{\mathrm{fd}}^{(\infty,0)}$ in Theorem \ref{luriehopkins} is obtained as follows. First, notice that $O(n)$ acts on the $n$-framings of a $k$-dimensional manifold $M$, and hence it gives an action on ${\rm Bord}^{fr}(n)$. Consequently, $O(n)$ acts on ${\rm Fun}^{\otimes}({\rm Bord}^{fr}(n),n\text{-}\mathrm{Vect})$. By the equivalence in eq. (\ref{cobframed}), we obtain an induced action of $O(n)$ on $(n\text{-}\mathrm{Vect})_{\mathrm{fd}}^{(\infty,0)}$ and so, for any homomorphism $G\to{O(n)}$, we have a corresponding $G$-action on $(n\text{-}\mathrm{Vect})_{\mathrm{fd}}^{(\infty,0)}$. The equivalence in eq. (\ref{Gcob}) is then obtained as a consequence of the equivalence between ${\rm Fun}^{\otimes}({\rm Bord}^{fr}(n), \allowbreak n\text{-}\mathrm{Vect})^{G}$ and ${\rm Fun}^{\otimes}({\rm Bord}^{G}(n), n\text{-}\mathrm{Vect})$.
\end{remark}

\begin{example}\label{2-tier}
A fully extended 2-dimensional oriented TQFT $Z$ is the datum of a semisimple Frobenius algebra $A$. To the oriented point ${\rm pt}^{+}$ it is assigned the linear category $\mathrm{Mod}_A$ of finite dimensional right $A$-modules, while the closed oriented 1-manifold $S^1$ is sent to the center of $A$, which is a commutative Frobenius algebra. See \cite{schom} for details. This is consistent with what one should have expected: the looped TQFT $\Omega Z$ is a 2-dimensional Atiyah-Segal TQFT, and these are equivalent to the category of commutative Frobenius algebras; see \cite{kock}. Note, however, that not every 2-dimensional Atiyah-Segal TQFT is obtained a the looping of a fully extended 2-dimensional TQFT, as a commutative Frobenius algebra need not to be semisimple.
\end{example}
\begin{example}\label{groups}
As a particular case of Example \ref{2-tier}, one can show that to any finite group $G$ is associated an extended 2-dimensional TQFT $Z_G$, mapping ${\rm pt}^{+}$ to the category of finite dimensional representations of $G$, and $S^1$ to the algebra $\mathbb{K}[G/\!/G]$ of class functions on $G$. For a review, see \cite{lee}.
\end{example}
The cobordism hypothesis tells us that the $\infty$-category of fully extended $n$-dimensional TQFTs of moduli level 0 constitutes an $\infty$-groupoid. This is in general no longer true when the moduli level is higher than 0. In particular, this means that if $Z_1$ and $Z_2$ are two TQFTs with moduli level greater than 0, it is possible to have nontrivial (i.e., non-invertible) morphisms between $Z_1$ and $Z_2$, as in Example  \ref{to-be-expanded} below. This possibility will be particularly relevant in the forthcoming sections.
\begin{remark}
A useful mechanism to produce fully extended $n$-dimensional TQFTs of moduli level 1 is to start from a fully extended $(n+1)$-dimensional TQFT of moduli level 0 and consider a truncation, as in Remark \ref{trunc}. If 
$Z_1$ and $Z_2$ are moduli level 0 fully extended $(n+1)$-dimensional TQFTs and 
\begin{equation}
\eta\colon Z_1\bigr\vert_{n}\to Z_2\bigr\vert_{n}
\end{equation}
is a morphism between their $n$-dimensional truncations, then, due to the cobordism hypothesis, $\eta$ will not in general lift to a morphism between $Z_1$ and $Z_2$. At the level of fully extended $(n+1)$-dimensional TQFTs, the morphism $\eta$ can be considered as a \emph{codimension 1 defect}, also known as a \emph{domain wall}. 
\end{remark}
\begin{example}\label{to-be-expanded}
Let $\underline{1}:{\rm Bord}^{or}(2)\to 2\text{-}\mathrm{Vect}$ be the trivial extended 2-dimensional oriented TQFT, which assigns to the oriented point the linear category of finite dimensional vector spaces, to $S^1$ the vector space $\mathbb{K}$, and to closed 2-manifolds the element $1$ in $\mathbb{K}$. Let $Z_G$ be the 2-tier extended 2-dimensional oriented TQFT associated with a finite group $G$, see Example \ref{groups}. Then, a morphism $\rho:\underline{1}\bigr\vert_{1}\to Z_{G}\bigr\vert_{1}$ is the datum of a finite dimensional representation $\rho$ of $G$, and in the fully extended 2-dimensional TQFT ``with defects'' lifting it, the representation $\rho$ becomes a domain wall and the cylinder \\
\begin{center}
\def\svgwidth{150pt}
\begingroup%
  \makeatletter%
  \providecommand\color[2][]{%
    \errmessage{(Inkscape) Color is used for the text in Inkscape, but the package 'color.sty' is not loaded}%
    \renewcommand\color[2][]{}%
  }%
  \providecommand\transparent[1]{%
    \errmessage{(Inkscape) Transparency is used (non-zero) for the text in Inkscape, but the package 'transparent.sty' is not loaded}%
    \renewcommand\transparent[1]{}%
  }%
  \providecommand\rotatebox[2]{#2}%
  \ifx\svgwidth\undefined%
    \setlength{\unitlength}{210.99199525bp}%
    \ifx\svgscale\undefined%
      \relax%
    \else%
      \setlength{\unitlength}{\unitlength * \real{\svgscale}}%
    \fi%
  \else%
    \setlength{\unitlength}{\svgwidth}%
  \fi%
  \global\let\svgwidth\undefined%
  \global\let\svgscale\undefined%
  \makeatother%
  \begin{picture}(1,0.50541584)%
    \put(0,0){\includegraphics[width=\unitlength]{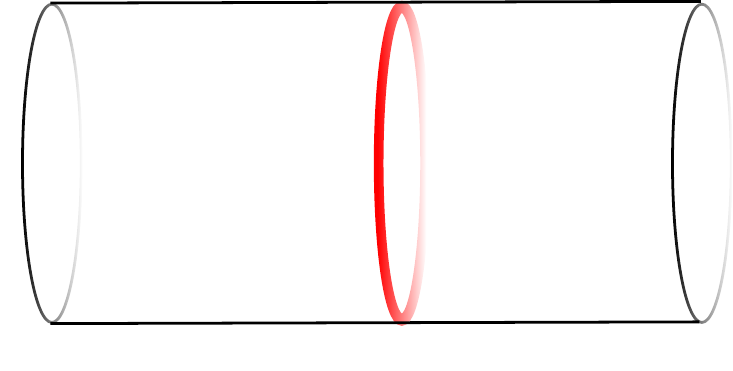}}%
    \put(-0.00001,0.01114527){\color[rgb]{0,0,0}\makebox(0,0)[lb]{\smash{$1$}}}%
    \put(0.53271732,0.01877767){\color[rgb]{0,0,0}\makebox(0,0)[lb]{\smash{$\rho$}}}%
    \put(0.93277679,0.00997219){\color[rgb]{0,0,0}\makebox(0,0)[lb]{\smash{G}}}%
  \end{picture}%
\endgroup%
\end{center}
corresponds to the character of $\rho$. The cylinder equipped with a circle defect depicted above appears in the literature with the name of \emph{transmission functor}, and plays an important role in the study of symmetries of topological quantum field theories \cite{fpsv}.
\end{example}
Since from the literature we are not aware of the any characterization of fully extended TQFTs with moduli level greater than 0, we conclude this section with a conjecture.
\begin{conjecture}[Cobordism hypothesis for TQFTs with moduli level $m$] For any $m\geq 0$
there is an equivalence of $\infty$-categories
\begin{equation}
{\rm Fun}^{\otimes}({\rm Bord}^{G}(n),(m+n)\text{-}\mathrm{Vect})\simeq((m+n)\text{-}\mathrm{Vect})_{\mathrm{fd}}^{G\, (\infty,m)}
\end{equation}
induced by the evaluation functor $Z\to Z({\rm pt}^{+})$. 
\end{conjecture}
In the above conjecture $((m+n)\text{-}\mathrm{Vect})_{\mathrm{fd}}^{G\, (\infty,m)}$ denotes the $(\infty,m)$ groupoid obtained from $((m+n)\text{-}\mathrm{Vect})_{\mathrm{fd}}^{G}$ by discarding all non-invertible $k$-morphisms with $k>m$.
\begin{example}
As a supporting evidence for the above conjecture, let us expand Example \ref{to-be-expanded} above. In the same notations as in Example \ref{to-be-expanded}, we have seen that any finite dimensional representation $\rho$ of $G$ gives rise to a 1-morphism $F_\rho$ between the moduli level 1 1-dimensional TQFTs  $\underline{1}\bigr\vert_{1}$ and  $Z_{G}\bigr\vert_{1}$. From conjecture, we should expect that a morphism of representations $f\colon \rho_1\to \rho_2$ induces a 2-morphism $F_{\rho_1}\to F_{\rho_2}$ if and only if $f$ is an isomorphism. This is actually true: looking at the data associated with the 1-dimensional manifold $S^1$,  we see that $F_{\rho_1}\to F_{\rho_2}$  induces a morphism in $\mathbb{K}[G/\!/G]$ between the character of $\rho_1$ and the character of $\rho_2$. But since the only morphisms in the vector space $\mathbb{K}[G/\!/G]$ (seen as a 0-category) are identities, this means that the representations $\rho_1$ and $\rho_2$ have the same character, and therefore they are isomorphic.
\end{example}
\section{Anomalies in Topological Quantum Field Theories}\label{sec.anomaly}

We consider now a particular type of TQFT, called \emph{invertible}, which will be relevant in the description of anomalies we present later.
\subsection{Invertible TQFTs}
To be able to define invertible TQFTs, we first need to introduce the following
\begin{definition}\label{picard}
 The Picard $\infty$-groupoid ${\rm Pic}(n\mbox{-}{\rm Vect})$ is defined as the $\infty$-category with objects given by the invertible objects in $n\mbox{-}{\rm Vect}$, and $k$-morphisms given by the invertible $k$-morphisms for any $k$.
\end{definition}
Notice that the Picard $\infty$-groupoid ${\rm Pic}(n\mbox{-}{\rm Vect})$ is a symmetric monoidal $(\infty,n)$-subcategory of $n\mbox{-}{\rm Vect}$. Moreover, Definition \ref{picard} can be extended to any symmetric monoidal $(\infty, n)$-category $\mathcal{C}$.
\begin{example} The Picard groupoid ${\rm Pic}(0\mbox{-}{\rm Vect})$ is the group $\mathbb{K}^*$ of invertible elements of the field $\mathbb{K}$, and identities as morphisms. The Picard groupoid  ${\rm Pic}(1\mbox{-}{\rm Vect})$ is the groupoid with objects given by complex vector spaces of dimension 1, 1-morphisms given by invertible linear maps, and identities for $k$-morphisms, for $k>1$. The Picard 2-groupoid ${\rm Pic}(2\mbox{-}{\rm Vect})$ can be realized as the 2-groupoid with objects given by ${\rm Vect}$-module categories of rank 1, 1-morphisms given by invertible module functors, 2-morphisms given by invertible module natural transformation, and identities for higher $k$-morphisms. See \cite{eno}.
\end{example}
An invertible TQFT is essentially an $\infty$-functor assigning objects to invertible objects, and morphisms to invertible morphisms. More precisely
\begin{definition}
An $n$-dimensional Topological Quantum Field Theory extended to codimension $k$ and with moduli level $m$ 
\begin{equation}
Z\colon\mathrm{Cob}_k^\infty(n)\to (m+k)\mbox{-}{\rm Vect}
\end{equation}
is said to be invertible iff it factors as
\begin{equation}
\xymatrix{\mathrm{Cob}_k^\infty(n)
\ar[dr]\ar[r]^{Z}&(m+k)\mbox{-}{\rm Vect}\\
&\ar[u]{\rm Pic}((m+k)\mbox{-}{\rm Vect})}
\end{equation}
\end{definition}
From every symmetric monoidal $(\infty, n)$-category $\mathcal{C}$ one obtains a symmetric monoidal $(\infty, n+1)$-category $B\mathcal{C}$ by taking the $\infty$-category with a single object, and with $\mathcal{C}$ as the $\infty$-category of morphisms. It is immediate to see that $B{\rm Pic}(n\mbox{-}{\rm Vect})$ is naturally identified with the full subcategory of ${\rm Pic}((n+1)\mbox{-}{\rm Vect})$ on the tensor unit of $(n+1)\mbox{-}{\rm Vect}$. This gives a natural embedding
\begin{equation}
B{\rm Pic}(n\mbox{-}{\rm Vect})\hookrightarrow {\rm Pic}((n+1)\mbox{-}{\rm Vect}).
\end{equation}
This observation leads us to the following
\begin{definition}
An invertible TQFT with moduli level $m$ 
\begin{equation}
Z\colon \mathrm{Cob}_k^\infty(n)\to{\rm Pic}((m+k)\mbox{-}{\rm Vect})\hookrightarrow (m+k)\mbox{-}{\rm Vect}
\end{equation}
 is said to be \emph{semitrivialized} if it is given a factorization of $Z$ through $B{\rm Pic}((m+k-1)\mbox{-}{\rm Vect})$.
\end{definition}
\begin{remark}\label{semitrivializable}
For $m+k=1,2$ the inclusion $B{\rm Pic}((m+k-1)\mbox{-}{\rm Vect})\hookrightarrow {\rm Pic}((m+k)\mbox{-}{\rm Vect})$
is an equivalence of $(m+k)$-groupoids. Therefore, an invertible TQFT with moduli level $m$ can always be (non canonically) semitrivialized as soon as $m+k\leq 2$. It is presently not clear whether this result holds true for $m+k>2$.
\end{remark}
\begin{remark}
An important aspect of invertible TQFTs is that they can be described as maps of spectra. Namely, an invertible TQFT factorizes through the ``groupoid $\infty$-completion'' $|\mathrm{Cob}_k^\infty(n)|$, which can be proven to be a spectrum in low dimensions. See \cite{freed1,freed2} for details.\\ 
We will not push in this direction in the present article.

\end{remark}
\par
\subsection{Anomalies} Invertible TQFTs of moduli level 1 will be particularly relevant to the present work: they will indeed describe anomalies.
\begin{definition}
An $n$-dimensional \emph{anomaly} is an invertible TQFT of moduli level 1
\begin{equation}
W\colon \mathrm{Cob}_k^\infty(n)\to{\rm Pic}((k+1)\mbox{-}{\rm Vect})\hookrightarrow (k+1)\mbox{-}{\rm Vect}.
\end{equation}
\end{definition}
\begin{remark}\label{restriction}
A natural way of producing an $n$-dimensional anomaly is by \emph{truncating} a $(n+1)$-dimensional TQFT with moduli level 0, i.e., by considering the composition
\begin{equation}
\mathrm{Cob}_k^\infty(n)\hookrightarrow \mathrm{Cob}_{k+1}^\infty(n+1) \to{\rm Pic}((k+1)\mbox{-}{\rm Vect})
\hookrightarrow (k+1)\mbox{-}{\rm Vect}.
\end{equation}
\end{remark}
\begin{example}\label{semitrivializedW}
Let us make explicit the data of a semitirivialized $n$-dimensional anomaly for $k=1$. By definition, this is
a symmetric monoidal functor
\begin{equation}
W\colon \mathrm{Cob}_1^\infty(n)\to B{\rm Pic}(1\mbox{-}{\rm Vect})\hookrightarrow {\rm Pic}(2\mbox{-}{\rm Vect})\hookrightarrow 2\mbox{-}{\rm Vect}.
\end{equation}
Therefore, to each $n$-dimensional cobordism $M$ a complex line $W_{M}$ is assigned, together with an isomorphism $W_{M\circ{M'}}\simeq W_{M}\otimes W_{M'}$, whenever $M\circ{M'}$ exists. This isomorphism, which we denote with $\psi_{MM'}$, is part of the structure of $W$, and hence has to obey the natural coherence conditions. In particular, to the trivial cobordism $\Sigma\times [0,1]$ is assigned the complex vector space $\mathbb{C}$. 
\end{example}
\begin{remark}\label{charactanom}
Recall from Remark \ref{gammaenne} that $B\Gamma^{\infty}(\Sigma)$ denotes the $\infty$-groupoid associated to $\mathrm{Diff}(\Sigma)$,\footnote{Here we are omitting the explicit reference to the framing $G\to O(n)$: the manifold $\Sigma$ here is (as always in this article) endowed with a $G$-framing of its stabilised tangent bundle, and $\mathrm{Diff}(\Sigma)$ denotes the group of $G$-framing preserving diffeomorphisms of $\Sigma$.} namely $B\Gamma^{\infty}(\Sigma)$ is the connected component of $\Sigma$ in 
${\rm Cob}^{\infty}_{0}(n-1)$.  Let $W$ be as in Example \ref{semitrivializedW}. By the mapping cylinder construction, we have the $\infty$-functor
\begin{equation}
B\Gamma^\infty(\Sigma)\hookrightarrow \mathrm{Cob}_0^\infty(n-1)\hookrightarrow \mathrm{Cob}_1^\infty(n)\xrightarrow{} B{\rm Pic}(1\mbox{-}{\rm Vect})
\end{equation}
where the last arrow is given by the factorisation of $W$ through $B{\rm Pic}(1\mbox{-}{\rm Vect})$.
In the terminology of Section \ref{twochar}, $W$ gives rise to a 2-character for $\Gamma^\infty(\Sigma)$.
\end{remark}
We can now introduce the definition of anomalous TQFTs with given anomaly $W$. These are called $W$-twisted field theories in \cite{stolz-teichner} and relative field theories in \cite{ft}.
\begin{definition}\label{def-anomaly}
Let $W\colon \mathrm{Cob}_k^\infty(n)\to{\rm Pic}((k+1)\mbox{-}{\rm Vect})\hookrightarrow (k+1)\mbox{-}{\rm Vect}$ be an $n$-dimensional anomaly. An anomalous $n$-dimensional extended TQFT with anomaly $W$ is a morphism of $n$-dimensional TQFTs with moduli level 1
\begin{equation}
Z_W\colon \underline{1}\to W,
\end{equation}
where $\underline{1}\colon \mathrm{Cob}_k^\infty(n)\to(k+1)\mbox{-}{\rm Vect}$ is the trivial TQFT mapping all objects to the monoidal unit and all morphisms to identities.
\end{definition}

\begin{lemma}
Let $W$ be the trivial $n$-dimensional anomaly, i.e., let $W=\underline{1}$. Then an $n$-dimensional extended anomalous TQFT with anomaly $W$ is equivalent to an ordinary $n$-dimensional extended TQFT. 
\end{lemma}
\begin{proof}
Immediate from Lemma \ref{endo}. 
\begin{remark} Strictly speaking, we have defined above a TQFT with \emph{incoming} anomaly, and one could also consider \emph{outgoing} anomalies by taking morphisms $W\to \underline{1}$, see, e.g., \cite{ft}. Although this distinction is relevant, e.g., for oriented theories, where one can also have both kinds of anomalies at the same time, we will not elaborate on this here.
\end{remark}

To get the flavour of these TQFTs with anomaly, let us spell out the data of an $n$-dimensional TQFT with semitrivialized anomaly in the $k=1$ case. As expected, we obtain a structure resembling an $n$-dimensional TQFT a l\'{a} Atiyah-Segal, but with a ``twisting'' coming from the anomaly $W$. Namely, if 
\begin{equation}
W\colon \mathrm{Cob}_1^\infty(n)\to B{\rm Pic}(1\mbox{-}{\rm Vect})\hookrightarrow {\rm Pic}(2\mbox{-}{\rm Vect})\hookrightarrow 2\mbox{-}{\rm Vect}
\end{equation}
is a semitrivialized anomaly, then a morphism $Z_W\colon \underline{1}\to W$ consists of the following collection of data:
\begin{itemize}
\item[a)]To each closed $(n-1)$-dimensional manifold $\Sigma$ it is assigned a vector space $V_{\Sigma}$, with $V_{\emptyset}\simeq {\mathbb{K}}$ and with functorial isomorphisms $V_{\Sigma\sqcup{\Sigma'}}\simeq  V_{\Sigma}\otimes V_{\Sigma'}$;
\item[b)]To each cobordism $M$ between $\Sigma$ and $\Sigma'$ it is assigned a linear map $\varphi_{M}: W_{M}\otimes V_{\Sigma}\to V_{\Sigma'}$; for $M$ the trivial cobordism, the corresponding linear map is the natural isomorphism $\varphi_{\Sigma\times [0,1]}: \mathbb{K}\otimes V_{\Sigma}\to V_{\Sigma}$.
\end{itemize}
Moreover, these data satisfy the following compatibilities:
\begin{itemize}
\item[i)] Let  $f_{MM'}:M\to{M'}$ be a diffeomorphism fixing the boundaries between two cobordisms $M$ and $M'$ between $\Sigma$ and $\Sigma'$. Then the following diagram commutes:
\begin{equation}\label{anom1}
\xymatrix{
W_{M}\otimes{V_{\Sigma}}\ar[r]^-{\varphi_{M}}\ar[d]_{f_{MM'*}\otimes{\rm id}} & V_{\Sigma'}\\
 W_{M'}\otimes{V_{\Sigma}}\ar[ru]_{\varphi_{M'}}
}
\end{equation}
where $f_{MM'*}:W_M\to W_{M'}$ denotes the isomorphism induced by $f_{MM'}$. 
\item[ii)] For any cobordism $M$ between $\Sigma$ and $\Sigma'$, and $M'$ between $\Sigma'$ and $\Sigma''$, the following diagram commutes
\begin{equation}\label{anom2}
\xymatrix{
W_{M'}\otimes W_{M}\otimes V_{\Sigma}\ar[r]^-{{\rm id}\otimes \varphi_{M}}\ar[d]^{\wr}_{\psi_{M'M}\otimes\mathrm{id}} & \ar[d]^{\varphi_{M'}}W_{M'}\otimes  {V_{\Sigma'}}\\
W_{M'\circ M}\otimes V_{\Sigma}\ar[r]^-{\varphi_{M'\circ M}}& {V_{\Sigma''}}
}
\end{equation}
\end{itemize}
In general, an anomalous TQFT as defined above will give rise to projective representations of diffeomorphsims of closed manifolds. In order to give a precise statement, in the following section we will take a detour into projective representations of $\infty$-groups as homotopy fixed points of higher characters.\\
\end{proof}
\section{$n$-characters and projective representations}\label{twochar}
In this section we will introduce the notion of an $n$-character for an $\infty$-group (e.g., the Poincar\'e $\infty$-groupoid $\pi_{\leq \infty}(G_{\mathrm{top}})$ of a topological group $G_{\mathrm{top}}$), and its homotopy fixed points. This is a natural higher generalisation of the notion of a $\mathbb{C}^{*}$-group character. Hence, as a warm up, we will first discuss the case of a discrete group $G$, and show how this recovers the category of (finite dimensional) projective representations of $G$. This is well known in geometric representation theory, but since we are not able to point the reader to a specific treatment in the literature, we will provide the necessary amount of detail here.
\subsection{Discrete groups}\label{finite}
Let $G$ be a (discrete) group, and let $BG$ denote the 1-object groupoid with $G$ as group of morphisms, regarded as an $\infty$-groupoid with only identity $k$-morphisms for $k>1$.
\begin{definition}\label{2chardisc} A 2-character for $G$ with values in $\vect$ is a 2-functor
\begin{equation}
\rho: {BG} \to B\pic
\end{equation}
\end{definition}
Explicitly, a 2-character $\rho$ consists of a family of complex lines $W^\rho_{g}$, one for each $g\in{G}$, and  isomorphisms
\begin{equation}
\psi_{g,h}^\rho:{W^\rho_{g}\otimes{W^\rho_{h}}}\xrightarrow{\simeq}W^\rho_{gh}
\end{equation}
satisfying the associativity condition
\begin{equation}\label{assoc}
\psi_{gh,j}^\rho \circ (\psi_{g,h}^\rho\otimes\id)= \psi_{g,hj}^\rho\circ (\id\otimes\psi_{h, j}^\rho)
\end{equation}
for any $g, h, j\in{G}$. When no confusion is possible we will simply write $W_g$ for $W_g^\rho$ and $\psi_{g,h}$ for $\psi^\rho_{g,h}$.\\
For a given group $G$, 2-characters form a category, given by the groupoid $[B{G}, B\pic]$ of functors between $B{G}$ and $B\pic$, and their natural transformations. Explicitly, a morphism $\rho\to \tilde{\rho}$ is a collection of isomorphisms of complex lines $\xi_g\colon W_g \xrightarrow{\sim} \tilde{W}_g$ such that
\[
\psi_{g,h}\circ (\xi_g\otimes \xi_h)=\xi_{gh}\circ \psi_{g,h},
\]
for any $g,h\in G$.\\
\medskip

The assignment $W \to W\otimes(-)$ induces an equivalence of groupoids
\begin{equation}\label{eq.picaut}
\pic\simeq\aut,
\end{equation}
where $\aut$ denotes the groupoid of linear auto-equivalences of $\vect$, i.e. of linear invertible functors from $\rm Vect$ to itself.
As a consequence, a 2-character defines an action of $G$ by functors on the linear category $\vect$. As for any action of a group, we can investigate the structure of its fixed points. Since we are in a categorical setting, though, we can ask that points are fixed at most up to isomorphisms. This motivates the following
\begin{definition}\label{fixdef}
Let $\rho=\{W_g;  \psi_{g,h}\}$ be a 2-character for a (discrete) group $G$. A homotopy fixed point for $\rho$ is given by an object $V\in\vect$ and a family $\{\varphi_{g}\}_{g\in{G}}$ of isomorphisms
\begin{equation}
\varphi_{g}: {W_g}\otimes{V}\xrightarrow{\simeq}{V}
\end{equation}
satisfying the compatibility condition
\begin{equation}\label{hofixed}
\varphi_{gh}\circ (\psi_{g,h}\otimes \id) = \varphi_{g}\circ(\id\otimes\varphi_{h})
\end{equation}
\end{definition}
\begin{remark}
A convenient way to encapsulate the data in Definition \ref{fixdef} is the following. By using the equivalence (\ref{eq.picaut}), a 2-character $\rho$ induces a 2-functor $W:BG\to{2\mbox{-}{\rm Vect}}$, which assigns to the single object in $BG$ the category ${\rm Vect}$.\footnote{In other words, $W$ is a 2-representation of $G$ of rank 1.} If we denote by $\underline{1}$ the trivial 2-functor from $BG$ to ${2\mbox{-}{\rm Vect}}$, we have then that a homotopy fixed point is equivalently a morphism, i.e. a natural transformation of 2-functors, $\underline{1}\to W$.  
\end{remark}
\begin{remark} 
Homotopy fixed points for a given 2-character $\rho$ form a category in a natural way, which we denote with $\vect^{\rho}$. 
It is immediate to see that, up to equivalence, ${\rm Vect}^{\rho}$ depends only on the isomorphism class of $\rho$.
\end{remark}
In the following, we will show that 2-characters for a group $G$ are related to group 2-cocyles for $G$, and that homotopy fixed points are related to projective representations.\\
Recall that to a group $G$ we can assign its groupoid of group $2$-cocycles with values in $\mathbb{K}^{*}$, which we denote by $\underline{Z}^{2}_{grp}(G;\mathbb{K}^{*})$. This is, essentially by definition, the 2-groupoid 
$[BG, B^{2}\mathbb{K}^{*}]$ of 2-functors from $BG$ to $B^2{\mathbb{K}^*}$. Since $B^2{\mathbb{K}^*}$ is the simplicial set with a single 0-simplex, a single 1-simplex, 2-simplices indexed by elements in $\mathbb{K}^*$, and 3-simplices corresponding to those configurations of 2-simplices the indices of whose boundary faces satisfy the 2-cocycle condition, a 2-functor $F\colon BG\to B^2{\mathbb{K}^*}$
is precisely a group 2-cocycle on $G$ with coefficients in $\mathbb{K}^*$.
\par
The equivalence $B\mathbb{K}^*\xrightarrow{\simeq} \mathrm{Pic}(1\text{-}\mathrm{Vect})$ induces an equivalence $B^2\mathbb{K}^*\xrightarrow{\simeq} B\mathrm{Pic}(1\text{-}\mathrm{Vect})$, and so
an equivalence 
\begin{equation}\label{eq.equivfinite}
T: \underline{Z}^{2}_{grp}(G;\mathbb{K}^{*}) 
\xrightarrow{\simeq}[BG,B(\mathrm{Pic}(1\text{-}\mathrm{Vect}))]\\
\end{equation}
for any finite group $G$. In particular, every 2-cocycle $\alpha$ naturally induces (and is actually equivalent to) a 2-character $T(\alpha)$. Note that $W^{T(\alpha)}_g=\mathbb{K}$ for any $g\in G$. The morphisms $\psi_{g,h}^{T(\alpha)}:W^{T(\alpha)}_g\otimes W^{T(\alpha)}_h\xrightarrow{\simeq} W^{T(\alpha)}_{gh}$ are given by
\begin{equation}
W^{T(\alpha)}_g\otimes W^{T(\alpha)}_h=\mathbb{K}\otimes\mathbb{K}\cong \mathbb{K}\xrightarrow{\alpha(g,h)}\mathbb{K}=W^{T(\alpha)}_{gh}.
\end{equation}
\medskip

Recall that a \emph{projective representation} for a group $G$ with 2-cocycle $\alpha$ is given by a vector space $V$, and a family of isomorphisms
\begin{equation}
\varphi^\alpha_{g}:V\xrightarrow{\simeq}V, \quad \forall g \in G
\end{equation}
satisfying the condition
\begin{equation}\label{projective}
\varphi^\alpha_{gh}=\alpha(g,h)\,\varphi^\alpha_{g}\circ\varphi^\alpha_{h}, \quad \forall g,h \in G 
\end{equation}
Projective representations for a given 2-cocycle $\alpha$ form naturally a category, which we denote with ${\rm Rep}^{\alpha}(G)$.\\
Given any projective representation $(V,\varphi^\alpha)$  with 2-cocycle $\alpha$, the vector space $V$ is naturally a homotopy fixed point for $T(\alpha)$: consider the family of isomorphisms
\begin{equation}
\varphi_{g}^{T(\alpha)}\colon W^{T(\alpha)}_g\otimes V=\mathbb{K}\otimes{V}\cong V\xrightarrow{\varphi_g^\alpha}V, \quad \forall g \in G.
\end{equation}
Then condition (\ref{projective}) assures that the family of isomorphisms $\{\varphi_{g}^{T(\alpha)}\}$ realises $V$ as a homotopy fixed point for $T(\alpha)$. It is immediate to check that this construction is functorial and therefore defines a ``realisation as homotopy fixed point'' functor $H^{\alpha}\colon{\rm Rep}^{\alpha}(G)  \to \vect^{T(\alpha)}$, for any 2-cocycle $\alpha$.
\begin{lemma}\label{lemma.hrealization}
The functor $H^\alpha\colon{\rm Rep}^{\alpha}(G)  \to \vect^{T(\alpha)}$
is an equivalence of categories.
\end{lemma}
\begin{proof}
It is immediate to see that $H^\alpha$ is faithful and full. To see that it is essentially surjective, take a homotopy fixed point $(V,\varphi)$ for $T(\alpha)$, and define $\varphi^\alpha$ as
\begin{equation}
\varphi^\alpha_g\colon V\cong \mathbb{K}\otimes V=W^{T(\alpha)}_g\otimes V\xrightarrow{\varphi_g}V.
\end{equation}
 Then the compatibility condition (\ref{hofixed}) ensures then that $(V,\varphi^\alpha)$ is a projective representation with 2-cocycle $\alpha$, with $H^\alpha(V,\varphi^\alpha)\simeq (V,\varphi)$.
\end{proof}
\subsection{2-characters for $\infty$-groups}
In this subsection we will see how the notion of a 2-character for a finite group immediately generalises to the notion of $(n+1)$-character for an $\infty$-group (i.e., for a monoidal $\infty$-groupoid whose objects are invertible for the monoidal structure) $G$, for any $n\geq 0$.\\
Since an $\infty$-group $G$ is in particular a monoidal $\infty$-category, it has a classifying monoidal $\infty$-category $BG$. The fact that $G$ is not just any monoidal $\infty$-category but an $\infty$-group can then be expressed by saying that $BG$ is a one-object $\infty$-groupoid.
The $\infty$-group structure on $G$ induces a (discrete) group structure on the set $\pi_0(G)$ of the isomorphism classes of objects of $G$, and one has a natural equivalence of groupoids  $B\pi_0(G)\cong \pi_{\leq 1}BG$.

\begin{example}
The basic example of an $\infty$-group is the fundamental $\infty$-groupoid of a topological group $G_{\mathrm{top}}$. Namely, since $G_{\mathrm{top}}$ is a group, the $\infty$-groupoid $\pi_{\leq \infty}(G_{\mathrm{top}})$ has a natural monoidal structure for which all the objects are invertible, given by the product in $G_{\mathrm{top}}$. Moreover, one has $\pi_0(\pi_{\leq \infty}(G_{\mathrm{top}}))=\pi_0(G_{\mathrm{top}})$, the (discrete) group of (path-)connected components of the topological group $G_{\mathrm{top}}$.
\end{example}
\begin{example}
A second fundamental example of an $\infty$-group is the $\infty$-group $\Gamma^\infty(\Sigma)$ of diffeomorphisms of a smooth manifold $\Sigma$. Here the objects are the diffeomorphisms of $\Sigma$, 1-morphisms are isotopies between diffeomorphism, 2-morphism are isotopies between isotopies, and so on. For oriented manifolds one can analogously consider the $\infty$-group of oriented diffeomorphisms, and more generally for $G$-framed manifolds one can consider the $\infty$-group of $G$-framed diffeomorphisms. The $\pi_0$ of the $\infty$-group $\Gamma^\infty(\Sigma)$ is the mapping class group $\Gamma(\Sigma)$ of the ($G$-framed) manifold $\Sigma$.
\end{example}
\begin{definition} 
Let $G$ be an $\infty$-group. A $n+1$-character for $G$ is a $\infty$-functor
\begin{equation}
\rho: BG \to B({\rm Pic}(n\mbox{-}{\rm Vect}))
\end{equation}
\end{definition}
The definition given above is very flexible and compact, and can be easily generalised by taking an arbitrary symmetric monoidal $(\infty,n)$-category in place of $n$-Vect.
\begin{remark}\label{2car} A 2-character for an $\infty$-group contains (in general) more information than a 2-character for a discrete group (which can be seen as a very particular case of an $\infty$-group). Namely,
for $G$ an $\infty$-group, a 2-character $\rho$ is given by an assignment to each object $g\in G$ of a complex line $W_{g}$, of a family $\psi_{g,h}$ of isomorphisms 
\begin{equation}
\psi_{g,h}:{W_{g}\otimes W_{h}} \xrightarrow{\sim}W_{gh}, \quad\forall{g,h}\in{G}
\end{equation}
and of isomorphisms
\begin{equation}
\psi_{f}:W_{g}\to{W_{h}}
\end{equation}
for any path (i.e., 1-morphism) $f$ connecting $g$ to $h$. The above isomorphisms must obey coherence conditions which encode the fact that $\rho$ is an $\infty$-functor. In particular, the isomorphism $\psi_f$ depends only on the isomorphism class of the 1-morphism $f$. In the particular case of a discrete group, the only paths in $G$ are the identities and one is reduced to Definition \ref{2chardisc}.
\end{remark}
\begin{example}\label{example.lie}
Let $G_{\mathrm{Lie}}$ be a Lie group, and let $L$ be a multiplicative line bundle over $G_{\mathrm{Lie}}$, equipped with a compatible flat connection $\nabla$. From $L$ one obtains a 2-character $\rho$ for $\pi_{\leq \infty}(G_{\mathrm{Lie}})$ as follows: to each $g$ in $G_{\mathrm{Lie}}$, one assigns the vector space given by the fiber $L_{g}$, and for each path $\gamma$ connecting $g$ and $h$ one takes the isomorphism $\psi_\gamma:L_{g}\to{L_{h}}$ induced by the connection via parallel transport (this depends only the homotopy class of $\gamma$, since $\nabla$ is flat). Finally, the fact that $L$ is multiplicative and the compatibility of $\nabla$ with the multiplicative structure imply that this assignment does define a 2-character. 
 \end{example}
\medskip

For any $n$, the $(n+1)$-group $\mathrm{Pic}(n\text{-}\mathrm{Vect})$ acts $(n+1)$-linearly on $n\text{-}\mathrm{Vect}$. This means that any $(n+1)$-character $\rho\colon BG\to B\mathrm{Pic}(n\text{-}\mathrm{Vect})
$ can naturally be seen as an $\infty$-functor $W\colon BG \to (n+1)\text{-}\mathrm{Vect}$, mapping the unique object of $BG$ to $n\text{-}\mathrm{Vect}$. We 
will denote by $\underline{1}\colon BG\to (n+1)\text{-}\mathrm{Vect}$ the trivial $\infty$-functor, mapping the unique object of $BG$ to the monoidal unit of $(n+1)\text{-}\mathrm{Vect}$ (i.e., to $n\text{-}\mathrm{Vect}$), and all morphisms in $BG$ to identities. 
\par
Having introduced this notation, we can give the following definition of homotopy fixed point for an $(n+1)$-character, generalizing the definition of homotopy fixed points for a 2-character of a discrete group seen above.
\begin{definition}
Let $\rho$ be an $(n+1)$-character for an $\infty$-group $G$, and let $W\colon BG\to (n+1)\text{-}\mathrm{Vect}$ be the corresponding $\infty$-functor. A homotopy fixed point for $\rho$ is a morphism of $\infty$-functors $\underline{1}\to W$.
\end{definition}
Homotopy fixed points for a $(n+1)$-character $\rho$ form naturally an $n$-category, which we denote by $n$-${\rm Vect}^{\rho}$.
\begin{remark}\label{remark-homotopy-fixed}
Since a 2-character for a $\infty$-group contains more information than a 2-character for a discrete group (see Remark \ref{2car}), being a homotopy fixed point is a more restrictive condition (in general) in the $\infty$-group case. Namely, with respect to the compatibility conditions in Definition \ref{fixdef}, one has in addition that the following diagram  
\begin{equation}\label{forcesholonomyzero}
\xymatrix{
W_{g}\otimes{V}\ar[r]^{\varphi_{g}}\ar[dr]_{\psi_f\otimes{\rm id}} & V\\
& W_{h}\otimes{V}\ar[u]_{\varphi_{h}} 
}
\end{equation}
has to commute, for any two objects $g$ and $h$ in $G$ and any 1-morphism $f\colon g\to h$ between them.
\end{remark}
\begin{remark}
Homotopy fixed points for a 2-character for a topological group are a special case of the following construction. Let $X$ be a $\infty$-groupoid, and let $L$ be a $\infty$-functor from $X$ to $B({\rm Pic(1\text{-}{\rm Vect})})$. A \emph{module} for $L$ is given by an $\infty$-functor $E:X\to{\rm Vect}$, and isomorphisms $L_{f}\otimes{E_{x}}\simeq{E_{y}}$ for any 1-morphism $f:x\to{y}$, where $L_{f}$ is the complex line assigned to $f$, and $E_{x}$ is the vector space assigned to $x$ by $E$. Higher morphisms must also be taken into account, and together with the above family of isomorphisms they must obey natural coherence conditions. The case of a homotopy fixed point for a 2-character for a topological group $G$ corresponds to $X=BG$. Another geometrically interesting case is when $X$ is the groupoid $Y^{[2]}\rightrightarrows{Y}$ for a surjective submersion $Y\to{M}$: in this case an $\infty$-functor $L\colon X\to B({\rm Pic(1\text{-}{\rm Vect})})$ is given by a bundle gerbe with a flat connection over $X$, while a module $E$ over $L$ is given by a (flat) gerbe module over $L$.  
\end{remark}

If $G$ is a (discrete) group and $\rho$ is a 1-character, i.e., a group homomorphism $G\to \mathbb{K}^*$, a homotopy fixed point is then nothing but a fixed point for the natural linear action of $G$ on $\mathbb{K}$ via $\rho$. Notice how the existence of a nonzero fixed point imposes a very strong constraint on the character $\rho$ in this case: if there exists a nonzero fixed point, then $\rho$ is the trivial character.
\par
An analogous phenomenon happens for $(n+1)$-characters of $\infty$-groups, for any $n\geq 0$. Here we will investigate in detail the case of 2-characters, due to its relevance to anomalous TQFTs. To do this, it is convenient to introduce the following terminology: we say that a 2-character $\rho\colon BG\to B{\rm Pic}(1\text{-}\mathrm{Vect})$ has \emph{trivial holonomy} if it factors through the natural projection $BG\to B\pi_0(G)$. The origin of this terminology is clear from Example \ref{example.lie}. There, the 2-character $\rho$ factors through $B\pi_{\leq \infty}(G_{\mathrm{Lie}})\to B\pi_0(G_{\mathrm{Lie}})$ precisely when the connection $\nabla$ has trivial holonomy.
We have then the following 
\begin{lemma}\label{holzero}
Let $V$ be a non-zero homotopy fixed point for a 2-character $\rho$. Then $\rho$ has trivial holonomy.
\end{lemma}
\begin{proof}
Since $V$ is a homotopy fixed point for $\rho$, by Remark \ref{remark-homotopy-fixed} we have the commutative diagram (\ref{forcesholonomyzero})
for any 2-morphism $f\colon g\to{h}$ in $BG$ (i.e., for any 1-morphism $f\colon g\to{h}$ in $G$). Since $\varphi_{g}$ and $\varphi_{h}$ are isomorphisms, we have
\begin{equation}
\psi_f\otimes \mathrm{id}=\varphi_h^{-1}\circ \varphi_g,
\end{equation}
and so $\psi_f\otimes \mathrm{id}$ is independent of $f$. Since $V$ is nonzero, this implies that $\psi_f$ is actually independent of $f$. This means that all the complex lines $W_g$ with $g$ ranging over a connected component (i.e., an isomorphism class of objects) of $G$ are canonically isomorphic to each other, and so $\rho$ factors through $B\pi_0(G)$. 
\end{proof}
Summing up the results in this section, we have the following
\begin{proposition}\label{ncarmain}
Let $\rho$ be a 2-character on an $\infty$-group $G$, and let $V$ be a nontrivial homotopy fixed point for $\rho$. Then there exist a 2-cocycle $\alpha_\rho$ on $\pi_0(G)$, unique up to equivalence, such that $V$ is isomorphic to (the homotopy fixed point realisation of) a projective representation of $\pi_0(G)$ with 2-cocycle $\alpha_\rho$.
\end{proposition}
\begin{proof}
Since $\rho$ has a nontrivial homotopy fixed point, $\rho$ has trivial holonomy by Lemma \ref{holzero}. Therefore, by definition of trivial holonomy, $\rho$ is (equivalent to) a 2-character on the discrete group $\pi_0(G)$. The statement then follows from Equation (\ref{eq.equivfinite}) and Lemma \ref{lemma.hrealization}.
\end{proof}
\subsection{Projective representations from TQFTs} We can finally apply the results on $(k+1)$-characters to anomalous TQFTs. 
Indeed, consider a semitrivialized anomaly $W\colon  \mathrm{Cob}_k^\infty(n)\to B{\rm Pic}(k\text{-}{\rm Vect})\hookrightarrow {(k+1)\mbox{-}{\rm Vect}}$, and let $Z_{W}$ be
an $n$-dimensional anomalous TQFT extended down to codimension $k$, with anomaly $W$. Reasoning as in Remark \ref{charactanom}, the anomaly $W$ induces, for any closed (oriented) $(n-k)$-dimensional manifold $\Sigma$, a 2-character $\rho_\Sigma$ for the $\infty$-group of (oriented) diffeomorphisms $\Gamma^\infty(\Sigma)$, as in the
following diagram

\begin{equation}\label{diagram}
\xymatrixcolsep{2pc}
\xymatrix{
B\Gamma^\infty(\Sigma)\ar@/_3pc/[rrr]_{\rho_\Sigma}
\ar@{^{(}->}[r] & \mathrm{Cob}_0^\infty(n-k)\ar@{^{(}->}[r] & \mathrm{Cob}_k^\infty(n)\ar[r]  \ar@/^2pc/[rr]^{W}
& B{\rm Pic}(k\text{-}{\rm Vect})\ar[r]& {(k+1)\mbox{-}{\rm Vect}}
}
\end{equation}
The $k$-vector space $Z_W(\Sigma)$ associated by the anomalous TQFT $Z_W$ to the (oriented) $(n-k)$-dimensional manifold $\Sigma$ is, by definition, a homotopy fixed point for $\rho_\Sigma$. In particular, for $k=1$, by Proposition \ref{ncarmain}, the vector space $Z_W(\Sigma)$ associated to an $(n-1)$-dimensional manifold $\Sigma$ is a projective representation of the mapping class group $\Gamma(\Sigma)$ as soon as $Z_W(\Sigma)$ is nonzero. In other words, for any $(n-1)$-dimensional manifold $\Sigma$ we obtain a central extension
\[
1\to \mathbb{K}^*\to \widetilde{\Gamma}(\Sigma)\to \Gamma(\Sigma) \to 1
\]
and a linear representation $\widetilde{\Gamma}(\Sigma)\to \mathrm{Aut}(Z_W(\Sigma))$. This can be neatly described by noticing that for $k=1$ the data for an anomalous TQFT with anomaly $W$ are a homotopy commutative diagram of the form
\[
\xymatrix{
 \mathrm{Cob}_1^\infty(n)\ar[r]\ar[d]_{W} &\mathbf{1}\ar[d] \\
B{\rm Pic}({\rm Vect}) \ar[r]& 2\mbox{-}{\rm Vect}
 \ar@{=>}(12.25,-5.25);(8.75,-8.75)^{Z_W}
}.
\]
Such a diagram can be interpreted as the datum of a section $Z_W$ of the 2-line bundle $\mathcal{L}$ over $\mathrm{Cob}_1^\infty(n)$ associated with $W$. The ``graph'' of this section is a $\infty$-category 
$\widetilde{\mathrm{Cob}}_1^\infty(n)$ over $\mathrm{Cob}_1^\infty(n)$ whose objects are pairs consisting of an $(n-1)$-dimen\-sional manifold $\Sigma$ together with the choice of an object in the fibre $\mathcal{L}_\Sigma$. The mapping class group for such a pair is the $\mathbb{K}^*$-central extension of $\Gamma(\Sigma)$ described above. Notice the striking similarity with Segal's description of projective modular functors via central extensions of the cobordism category \cite{segal}, with the remarkable difference that anomalies in the sense of the present article induce $\mathbb{K}^*$-central extensions whereas in Segal's extended cobordism one deals with $\mathbb{Z}$-central extensions. 
\begin{remark}
As we have seen above, having a semitrivialized anomaly $W$ produces projective representations of the mapping class groups of all closed $(n-k)$-dimensional manifolds at once. If one is interested in a single  $(n-k)$-dimensional manifold $\Sigma$, though, there is no need for a semitrivialization of the anomaly: indeed, one can produce a projective representation of $\Gamma(\Sigma)$ from any anomalous TQFT $Z_W$, as soon as the invertible $(k+1)$-vector space $W(\Sigma)$ is equivalent to the ``trivial'' $(k+1)$-vector space $k\text{-}{\rm Vect}$. As already observed in Remark \ref{semitrivializable}, this is always possible, although non canonically, for any invertible $(k+1)$-vector space, with $k=0,1$. Namely, choosing an equivalence between $W(\Sigma)$ and $k\text{-}{\rm Vect}$ amounts to give a homotopy commutative diagram
\[
\xymatrix{
B\mathrm{Aut}(W(\Sigma)) \ar[r]^{W(\Sigma)}\ar[d]& (k+1)\mbox{-}{\rm Vect}\\
B{\rm Pic}(k\mbox{-}{\rm Vect})\ar[ru]
 \ar@{=>}(8.25,-3.25);(12.75,-6.75)^{\Psi}
},
\]
where the top horizontal arrow  picks the $(k+1)$-vector space $W(\Sigma)$, while the diagonal arrows is the canonical embedding of $B{\rm Pic}(k\mbox{-}{\rm Vect})$ into $(k+1)\mbox{-}{\rm Vect}$, which picks the $(k+1)$-vector space $k\text{-}{\rm Vect}$.
The construction of the projective representation of the mapping class group of $\Sigma$ follows from the very same arguments as above: indeed, just notice that in diagram (\ref{diagram}) it is inessential to have the arrow $\mathrm{Cob}_k^\infty(n)\to B{\rm Pic}(k\text{-}{\rm Vect})$ if we are interested in a single manifold $\Sigma$, while at the same time the morphism $B\Gamma^\infty(\Sigma)\to (k+1)\mbox{-}{\rm Vect}$ naturally factors through $B\mathrm{Aut}(W(\Sigma))$. We therefore obtain the following variant of diagram (\ref{diagram}), which induces the same considerations as above:
\begin{equation}
\xymatrixcolsep{2pc}
\xymatrix{
B\Gamma^\infty(\Sigma)\ar@/_1pc/[rrd]_{\rho_\Sigma}
\ar@{^{(}->}[r] & \mathrm{Cob}_0^\infty(n-k)\ar@{^{(}->}[r] & \mathrm{Cob}_k^\infty(n)  \ar@/^3pc/[rr]^{\underline{1}}\ar[rr]^{W}
& & {(k+1)\mbox{-}{\rm Vect}}\\
& &B\mathrm{Aut}(W(\Sigma))\ar[r] \ar[rru]^{W(\Sigma)}&B{\rm Pic}(k\text{-}{\rm Vect})\ar@/_1pc/[ur]
\ar@{=>}(100.25,-7.25);(104.75,-9.75)^{\Psi}
\ar@{=>}(90,9);(90,5)^{Z_W}
\ar@{=>}(40,-4.25);(50,-8.75)^{\sim}
}
\end{equation}
\end{remark}

\section{Boundary conditions for TQFTs}\label{boundary}
\subsection{Boundary conditions} The $n$-dimensional TQFTs defined in Section \ref{tqft} assign diffeomorphism invariants to \emph{closed} $n$-manifolds. Neverthless, $n$-manifolds with boundaries have also invariants, usually obtained via relative constructions. One possibility to incorporate invariants of manifolds with boundaries is to enlarge the cobordism category with morphisms represented by manifolds with \emph{constrained} boundaries. The guiding example is given by 2-dimensional open/closed topological field theory \cite{laudapfeiffer,lazaroiu,mooresegal}, 
where the authors enlarge the category $\mathrm{Cob}_1(2)=\pi_{\leq 1}\mathrm{Cob}_1^\infty(2)$ of 2-dimensional cobordism by adding to it 1- and 2-dimensional manifolds with part of the boundary declared to be constrained, meaning that it is not possible to glue along. If we denote by ${\rm Cob}_1^{\partial}(2)$ this enlarged category, we will have the following 1-manifolds (and disjoint union of) as objects
\begin{center}
\includegraphics[scale=0.7]{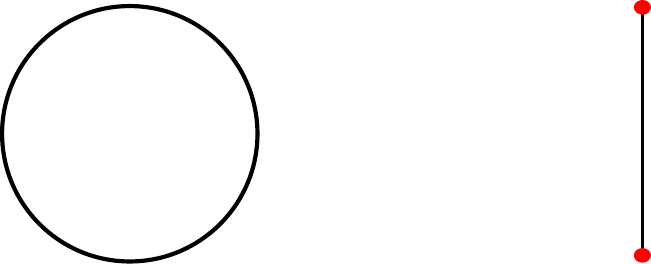}
\end{center}
and the following 2-manifolds as some of the morphisms
\begin{center}
\includegraphics[scale=0.6]{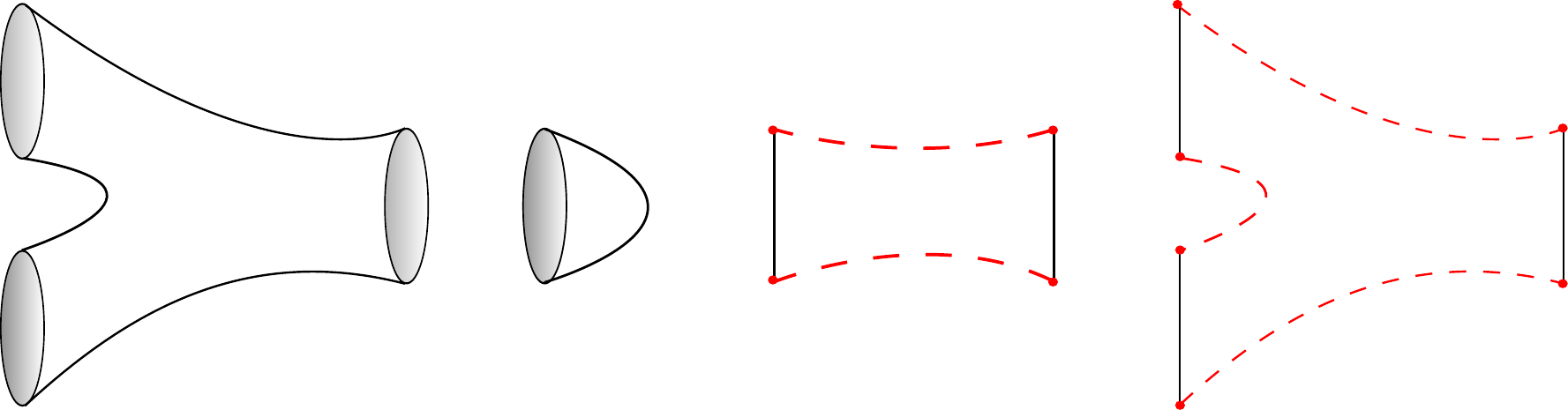}
\end{center}
where we denote the constrained boundary with a  dashed red line. Notice that, differently from \cite{mooresegal}, we are here using only \emph{one} type of constrained boundary, which we label/color red. The general case will be discussed in Remark \ref{moree-segal-setup} below.\\ 
\par
Inspired by the description of ${\rm Cob}^{\partial}(2)$ sketched above, let us define 
iteratively a \emph{constrained bordism} between two constrained $d$-dimensional manifolds $\Sigma_{0}$ and $\Sigma_{1}$ as a $(d+1)$-dimensional manifold \footnote{Here manifold more precisely means ``manifold with corners''.} $M$ whose boundary $\partial{M}$ can be decomposed as $\Sigma_{0}\cup{\Sigma_{1}}\cup{\partial_{const}{M}}$, where $\partial_{const}M$ is a cobordism from $\partial_{const}{\Sigma_{0}}$ to $\partial_{const}{\Sigma_{1}}$. Constrained cobordisms come with smooth collars around the part of the boundary which is unconstrained, in order to be able to glue them. With this premise, we can give the following informal definition, a rigorous version of which can be found in  \cite[Section 4.3]{lurie}.
\begin{definition}
The symmetric monoidal $(\infty,n)$-category $\mathrm{Bord}^\partial(n)$ has points as objects, 1-dimensional constrained bordisms as 1-morphisms, 2-dimensional constrained bordisms between constrained bordisms as 2-morphisms, and so on until we arrive at $n$-dimensional constrained bordisms as $n$-morphisms, from where higher morphisms are given by diffeomorphisms fixing the unconstrained boundaries and isotopies between these (and isotopies between isotopies, and so on).  
\end{definition}
\begin{remark}
Exactly as $\mathrm{Bord}(n)$, also $\mathrm{Bord}^\partial(n)$ comes in different flavours corresponding to the various possible $G$-framings of the cobordisms. In this section we will be interested in the general features of TQFTs with boundary conditions, and in their relation to anomalous field theories. Hence in what follows, we will always leave the $G$-marking unspecified, unless stated otherwise.
\end{remark}
\begin{example}\label{example-sguinci}
The following 1-dimensional constrained cobordisms are examples of 1-morphisms in $\mathrm{Bord}^{\partial,or}(n)$, for any $n\geq 1$.
\begin{figure}[h!]
\includegraphics[scale=0.4]{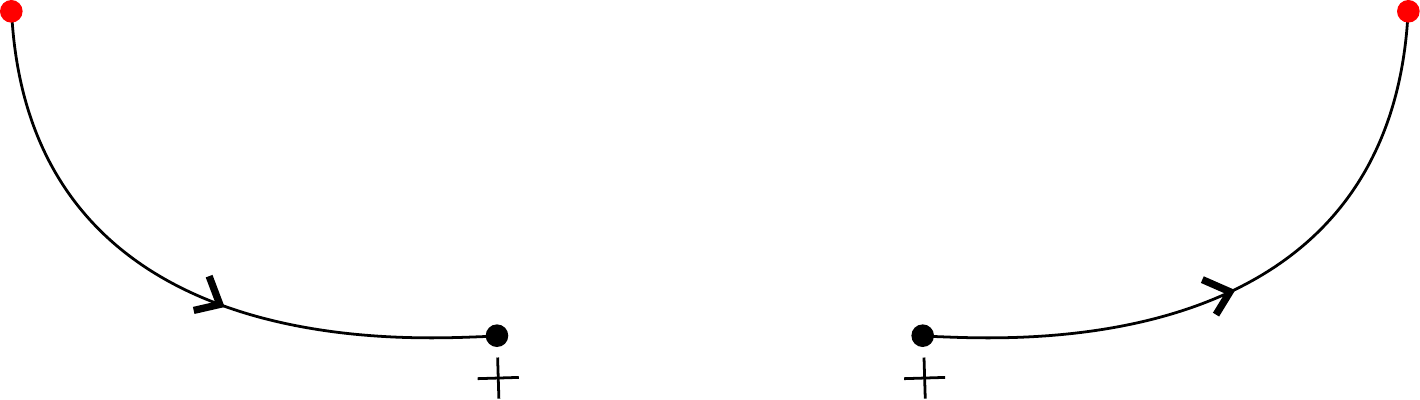}
\end{figure}\\
The one on the left represents a 1-morphism $\emptyset \to {\rm pt}^+$, which cannot be realized in  $\mathrm{Bord}(n)$. Similarly, the morphism on the right represents a 1-morphism ${\rm pt}^+\to \emptyset$, which is also not present in $\mathrm{Bord}(n)$. 
\end{example}
In analogy with the notation used in the unconstrained case, we will set
\begin{equation}
{\rm Cob}^{\partial,\infty}_k(n)=\Omega^{n-k}\mathrm{Bord}^\partial(n).
\end{equation}
With this notation, we have that the category of 2-dimensional constrained cobordism mentioned above is given by $\mathrm{Cob}^\partial_1(2)=\pi_{\leq 1}\mathrm{Cob}_1^{\partial,\infty}(2)$. 
There is a canonical (non full) embedding $\mathrm{Bord}(n)\hookrightarrow \mathrm{Bord}^\partial(n)$, hence for any $k\geq{0}$ we have a natural (non full) embeddings
\begin{equation}
i:{\rm Cob}^{\infty}_k(n)\hookrightarrow {\rm Cob}^{\partial,\infty}_k(n).
\end{equation}
This allows us to give the following
\begin{definition}
Let $Z\colon {\rm Cob}^{\infty}_k(n)\to (k+m)\text{-}\mathrm{Vect}$ be an $n$-dimensional TQFT wih moduli level $m$. A \emph{boundary condition} for $Z$ is a symmetric monoidal extension 
\begin{equation}
\xymatrix{{\rm Cob}^{\partial,\infty}_k(n)\ar[r]^{\tilde{Z}}&(k+m)\text{-}\mathrm{Vect}\\
{\rm Cob}^\infty_k(n)\ar[ur]_{Z}\ar[u]^{i}
}
\end{equation}
\end{definition}
\begin{remark}
It is important to notice that boundary conditions for an invertible TQFT are \emph{not} required to be invertible. This is reminiscent of the definition of an anomalous TQFT, where the morphism $\underline{1}\to W$ is not required to be an isomorphism. We will come back to this in Section \ref{boundanom}.
\end{remark}
\begin{example}\label{sguinci} The definition above can be made completely explicit for an Atiyah-Segal 1-dimensional TQFT, i.e., for $Z\colon {\rm Cob}^{\infty}_1(1)\to \mathrm{Vect}$. Indeed, in the same way as $Z$ factors through $\mathrm{Cob}_1(1)$, $\tilde{Z}$ will factor through $\mathrm{Cob}_1^\partial(1)=\pi_{\leq 1}{\rm Cob}^{\partial,\infty}_1(1)$. The objects of ${\rm Cob}^{\partial}(1)$ are oriented points, and the morphisms are given by those in ${\rm Cob}(1)$, and in addition the following constrained morphisms
\begin{center}
\includegraphics[scale=0.4]{sguinci.pdf}
\end{center}
and their duals. Therefore, if the 1-dimensional TQFT $Z$ is given by the finite-dimensional vector space $V$, then a boundary condition $\tilde{Z}$ for $Z$ is the datum of a pair $(v,\varphi)$, where $v$ is a vector in $V$ and $\varphi$ is an element in the dual space $V^*$. We will call these a \emph{left} and a \emph{right} boundary condition, respectively.
 In the unoriented situation the two morphisms above are identified, and a boundary condition reduces to the datum of the vector $v$, which also plays the role of a linear functional on $V$ via the symmetric nondegenerate inner product on $V$.
\end{example}
What makes the description of the boundary conditions so simple in the example above is the fact that we are dealing with a fully extended theory. Indeed, one has the following extension of the cobordism hypothesis to cobordisms with constrained boundaries \cite{lurie}.
\begin{theorem}[Lurie-Hopkins]\label{luriehopkins2}
Let $Z:{\rm Bord}^{fr}(n)\to n\text{-}\mathrm{Vect}$ be a fully extended TQFT with moduli level 0.
Then there is an equivalence
\begin{equation}
\{\text{(Left) boundary conditions for $Z$}\} \cong \mathrm{Hom}_{n\text{-}\mathrm{Vect}}((n-1)\text{-}\mathrm{Vect},Z({\rm pt}^+))\cong{Z({\rm pt}^+)}
\end{equation}
induced by the evaluation of $\tilde{Z}$ on the decorated interval on the left in Example \ref{example-sguinci}. 
\end{theorem}
This description of (left) boundary conditions is strongly reminescent of an anomalous TQFT as in Definition \ref{def-anomaly}. In the following we will see how a TQFT with (left) boundary conditions naturally induces an anomalous TQFT.
\begin{remark}
For TQFTs with values in an arbitrary symmetric monoidal $(\infty,n)$-category $\mathcal{C}$, one still has that the $(\infty, n-1)$-category of boundary conditions is equivalent to the hom-space $\mathrm{Hom}_{\mathcal{C}}(1_{\mathcal{C}},Z({\rm pt}^+))$, where $1_{\mathcal{C}}$ is the monoidal unit of $\mathcal{C}$. However in general this hom-space is not equivalent to $Z({\rm pt}^+)$.
\end{remark}
\begin{remark}\label{boundary-oriented}
An analogue statement is likely to hold for cobordisms with a reduction $G\to O(n)$ of the structure group of $n$-dimensional manifolds, by suitably taking into account the homotopy $O(n)$-action on the homotopy $G$-fixed point ${Z({\rm pt}^+)}$. For instance, in the oriented situation one has $O(n)/SO(n)=\mathbb{Z}/2\mathbb{Z}$, and the full boundary conditions data consist of a left boundary condition $(n-1)\text{-}\mathrm{Vect}\to Z({\rm pt}^+)$ and a right boundary condition $Z({\rm pt}^+)\to
(n-1)\text{-}\mathrm{Vect}$. Yet, for $n\geq 2$, every $n$-vector space $V$ realized as a linear $(n-1)$-category comes naturally equipped with a distinguished inner product given by the Hom bifunctor
\begin{equation}
\mathrm{Hom}:V^{op}\boxtimes V\to (n-1)\text{-}\mathrm{Vect}
\end{equation}
With this choice of inner product, left boundary conditions automatically determine right boundary conditions as in the unoriented case. 
\end{remark}
\begin{remark}\label{moree-segal-setup}One can consider more than a single boundary condition at once, by replacing ${\rm Bord}^{\partial}(n)$ by the larger symmetric monoidal $(\infty,n)$-category  ${\rm Bord}^{\partial_{J}}(n)$, where constrained boundaries are labelled by indices from a set $J$ of \emph{colours}. An extension $\tilde{Z}$ of a TQFT $Z$ to ${\rm Cob}_k^{\partial_J,\infty}(n)$ is then the assignment of a boundary condition to each colour $j\in J$, in such a way that the constraints imposed by requiring $\tilde{Z}$ to be a monoidal symmetric functor are satisfied. One can in particular make the tautological choice $
J=\text{objects}(\mathcal{B}_Z)$, where $\mathcal{B}_Z$ denotes the category of boundary conditions for $Z$.
In this way we recover the open/closed field theory framework as in \cite{laudapfeiffer,lazaroiu,mooresegal}. Namely, we recall from Example \ref{2-tier} that an extended 2-dimensional \emph{oriented} TQFT $Z$ is the datum of a semisimple Frobenius algebra $A$, to be seen as a placeholder for its category of finite dimensional right modules. 
Using the Hom functor as an inner product on ${}_A\mathrm{Mod}$ reduces boundary conditions to left boundary conditions (see Remark \ref{boundary-oriented}). Therefore one has constrained boundaries decorated by right $A$-modules, and the boundary condition $\tilde{Z}$ associates with the oriented segment with constrained boundaries
\begin{center}
\includegraphics[scale=0.4]{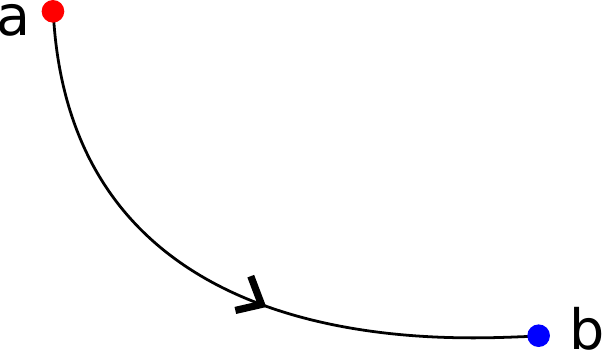}
\end{center}
decorated by the $A$-modules $R_a$ and $R_b$ the vector space $\mathcal{O}_{ab}=\mathrm{Hom}_A(R_a,R_b)$. See \cite{alexeevski-natanzon} for a treatment of open/closes 2d \emph{nonoriented} TQFTs.
\end{remark}

\begin{remark}
As an intermediate symmetric monoidal $(\infty,k)$-category between $\mathrm{Cob}_k^\infty(n)$ and $\mathrm{Cob}_k^{\partial,\infty}(n)$, one can consider the \emph{closed sector} $\mathrm{Cob}_{k,\mathrm{cl}}^{\partial,\infty}(n)$, defined as the \emph{full}  $(\infty,k)$-subcategory generated by $\mathrm{Cob}_k^\infty(n)$ inside $\mathrm{Cob}_k^{\partial,\infty}(n)$. Namely, objects in $\mathrm{Cob}_{k,\mathrm{cl}}^{\partial,\infty}(n)$ are closed $k$-manifolds, as in $\mathrm{Cob}_k^\infty(n)$.  Notice that in $\mathrm{Cob}_k^{\partial,\infty}(n)$ we allow for more objects, since one can consider $k$-manifolds with completely constrained boundary. For instance, of the two objects in $\mathrm{Cob}_1^{\partial,\infty}(2)$ depicted at the beginning of this section, only $S^1$ is an object in the closed sector.
\par
One can therefore also consider closed sector boundary conditions, i.e., extensions of a TQFT to the closed sector 
\begin{equation}
\xymatrix{{\rm Cob}^{\partial,\infty}_{k,\mathrm{cl}}(n)\ar[r]^{\tilde{Z}_{\mathrm{cl}}}&(k+m)\text{-}\mathrm{Vect}\\
{\rm Cob}^\infty_k(n)\ar[ur]_{Z}\ar[u]^{i}
}
\end{equation}
These are expected to be particularly simple in the $k=n-1$ case. Indeed, since $S^1$ is the only closed 1-dimensional manifold up to cobordisms, closed sector boundary conditions for a TQFT $Z\colon  {\rm Cob}^{\infty}_{n-1}(n)\to (n-1)\text{-}\mathrm{Vect}$ should reduce to a $(n-1)$-linear morphism $(n-2)\text{-}\mathrm{Vect}\to Z(S^1)$, i.e., to an object in $Z(S^1)$. This is in agreement with the findings in the literature on extended 3-dimensional TQFTs, where boundary decorations for a 2-dimensional surface $\Sigma$ with boundary components are objects in the modular tensor category the TQFT associates to $S^1$ \cite{bakalovkirillov}.
\end{remark}
\section{From boundary conditions to anomalous TQFTs}\label{boundanom}

As mentioned in the previous section, there is a close relation between boundary conditions for invertible TQFTs and anomalous TQFTs. In the present section we will exploit this relation in detail.\\
Let $\tilde{Z}$ be a boundary condition for an $(n+1)$-dimensional invertible TQFT $Z$ extended up to codimension $k+1$ with moduli level 0. In other words, we have the following commutative diagram
\begin{equation}
\xymatrix{{\rm Cob}^{\partial,\infty}_{k+1}(n+1)\ar[r]^{\tilde{Z}}&(k+1)\text{-}\mathrm{Vect}\\
{\rm Cob}^\infty_{k+1}(n+1)\ar[r]^-{Z}\ar[u]^{i}& \mathrm{Pic}((k+1)\text{-}\mathrm{Vect})\ar[u]
}
\end{equation}
As mentioned in Remark \ref{restriction}, the restriction of $Z$ to $\mathrm{Cob}_{k}(n)\hookrightarrow \mathrm{Cob}_{k+1}(n+1)$ is an $n$-dimensional anomaly, which we will denote $W^{Z}$.\\
Let $[\underline{0},1]$ denote the oriented interval $[0,1]$ with $\{0\}$ being a constrained component of the boundary, as in the figure in Example \ref{example-sguinci}, on the left. Then for any $m$-morphism $\Sigma$ in $\mathrm{Cob}_{k}(n)$, with $k\geq{0}$, i.e. for any $(n-k+m)$-dimensional manifold $\Sigma$, possibly with boundary, the product manifold $\Sigma\times [\underline{0},1]$ can be seen as a $(m+1)$-morphism from $\emptyset$ to $\Sigma$ in ${\rm Cob}^{\partial,\infty}_{k+1}(n+1)$:
\begin{equation}
\emptyset\xrightarrow{\Sigma\times [\underline{0},1]}\Sigma,
\end{equation}
We can graphically depict the morphism above as follows
\begin{center}
\def\svgwidth{150pt}
\begingroup%
  \makeatletter%
  \providecommand\color[2][]{%
    \errmessage{(Inkscape) Color is used for the text in Inkscape, but the package 'color.sty' is not loaded}%
    \renewcommand\color[2][]{}%
  }%
  \providecommand\transparent[1]{%
    \errmessage{(Inkscape) Transparency is used (non-zero) for the text in Inkscape, but the package 'transparent.sty' is not loaded}%
    \renewcommand\transparent[1]{}%
  }%
  \providecommand\rotatebox[2]{#2}%
  \ifx\svgwidth\undefined%
    \setlength{\unitlength}{391.88502692bp}%
    \ifx\svgscale\undefined%
      \relax%
    \else%
      \setlength{\unitlength}{\unitlength * \real{\svgscale}}%
    \fi%
  \else%
    \setlength{\unitlength}{\svgwidth}%
  \fi%
  \global\let\svgwidth\undefined%
  \global\let\svgscale\undefined%
  \makeatother%
  \begin{picture}(1,0.50058139)%
    \put(0,0){\includegraphics[width=\unitlength]{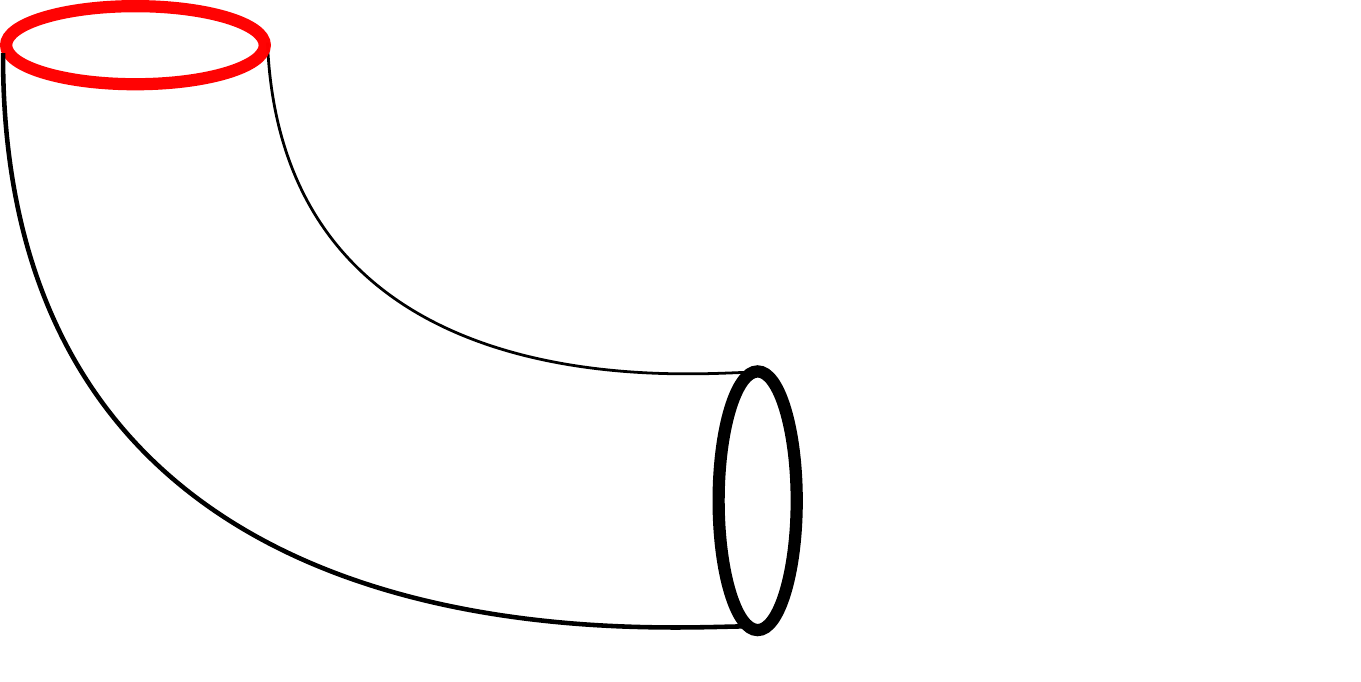}}%
    \put(0.22867604,0.44943832){\color[rgb]{0,0,0}\makebox(0,0)[lb]{\smash{$\Sigma\times{\underline{0}}$}}}%
    \put(0.60429641,0.00849261){\color[rgb]{0,0,0}\makebox(0,0)[lb]{\smash{$\Sigma\times{1}$}}}%
  \end{picture}%
\endgroup%
\end{center}
Moreover, given an $(n-k+m+1)$-cobordism $M$ between $\Sigma$ and $\Sigma'$, we have that the coloured manifold $M\times[\underline{0},1]$ induces a cobordism between $\Sigma\times[\underline{0},1]$ and $\Sigma'\times[\underline{0},1]$.\\
Evaluating $\tilde{Z}$ on $\Sigma\times[\underline{0},1]$ gives us a $(m+1)$-morphism in $(k+1)\text{-}\mathrm{Vect}$
between the unit (in the correct degree) and $\tilde{Z}(\Sigma)=Z(\Sigma)=W^{Z}(\Sigma)$.\\
Recall that a $(k+1)$-morphism in $\mathrm{Cob}_{k}(n)$ is a diffeomorphism $\varphi\colon\Sigma_1\to\Sigma_2$ of $n$-dimensional manifolds fixing the boundaries. By combining it with the identity of $ [\underline{0},1]$, one gets a diffeomorphism of $(n+1)$-dimensional manifolds, which realizes a $(k+2)$-morphism in  ${\rm Cob}^{\partial,\infty}_{k+1}(n+1)$ between the empty set and the mapping cylinder of $\varphi$. Applying $\tilde{Z}$ we get a morphism from the unit to $\tilde{Z}(M_\varphi)=Z(M_\varphi)=W^{Z}(\varphi)$. This pattern continues with no changes to isotopies between diffeomorphisms, isotopies between isotopies, etc. Hence we have that
\begin{equation}
 \tilde{Z}_{W^{Z}}:=\tilde{Z}(-\times [\underline{0},1])
 \end{equation}
 defines a morphism $\tilde{Z}_{W^{Z}}:\underline{1}\to W^{Z}$, i.e. an anomalous TQFT in the sense of Definition \ref{def-anomaly}.\\
We can assemble the argument above in the following
\begin{proposition}\label{boundaryanomalous}
Let $Z$ be a $(n+1)$-dimensional invertible TQFT extended down to codimension $k+1$ with moduli level 0, and let $W^{Z}$ denote the $n+1$-dimensional anomaly induced by $Z$. Then any boundary condition $\tilde{Z}$ for $Z$ induces an $n$-dimensional anomalous TQFT $\tilde{Z}_{W^{Z}}$  with anomaly $W^{Z}$. 
\end{proposition}

 The above argument shows that we have a ``forgetful map''
 \begin{equation}
 \{\text{boundary conditions on invertible TQFTs}\}\rightsquigarrow \{\text{anomalous TQFTs}\}
 \end{equation}
 In general, we do not expect the converse to hold. Namely, an anomalous TQFT with anomaly $W$ contains too little information to determine a boundary condition $\tilde{Z}$. 
Neverthless, in the case of fully extended TQFTs the situation is rather different.
\begin{remark}
The procedure of taking ``cartesian products'' with the constrained interval can be seen as a form of dimensional reduction for manifolds with boundaries. It is completely analogous to dimensional reduction over $S^{1}$, which allows to obtain a $n-1$-dimensional extended TQFT from an $n$-dimensional one, preserving the tiers of extension.  
\end{remark} 
\subsection{Boundary conditions for fully extended TQFTs} For simplicity, in the following we will consider the framed case. Let $Z$ be a $(n+1)$-dimensional fully extended invertible TQFT, namely an $\infty$-functor $Z: \mathrm{Bord}^{fr}(n+1) \to (n+1)\text{-}\mathrm{Vect}$ which factors through $\mathrm{Pic}((n+1)\text{-}\mathrm{Vect})$. As mentioned in Remark \ref{restriction}, from $Z$ we obtain an $n$-character $W^{Z}$. Let $Z_{W^{Z}}$ be an anomalous TQFT with anomaly $W^{Z}$, namely a morphism $\underline{1}\to{W^{Z}}$, which contains in particular the datum of a 1-morphism
\begin{equation}
{n\mbox{-}{\rm Vect}} \to W^{Z}({\rm pt}^{+})=Z({\rm pt}^{+})
\end{equation}
By Theorem \ref{luriehopkins2}, we have then that $Z_{W}$ induces a boundary condition $\tilde{Z}$ of $Z$, and an equivalence
\begin{equation}
Z_{W^{Z}}\simeq\tilde{Z}_{W^{Z}}
\end{equation}
of 1-morphisms $\underline{1}\to{W^{Z}}$, where $\tilde{Z}_{W^{Z}}$ is the anomalous TQFT as from Proposition \ref{boundaryanomalous}. This argument can be assembled in the following
\begin{theorem}
Let $Z$ be a fully extended invertible $(n+1)$-dimensional TQFT. Any $n$-dimensional anomalous TQFT $Z_{W^{Z}}$ with respect to $W^{Z}$ gives rise to a boundary condition $\tilde{Z}$ of $Z$.
\end{theorem}
Hence in the fully extended case, an anomalous TQFT with respect to an anomaly obtained by restriction of a higher dimensional TQFT $Z$ contains enough information to allow $Z$ to be extended on manifolds with boundaries.\\ 
\medskip

We conclude this section with an observation we find intriguing. In \cite{fhlt} a 4-category with duals ${\rm Braid}^{\otimes}$ of braided tensor categories has been introduced, as follows:
\begin{itemize}
\item[-] objects are given by braided tensor categories $\mathcal{C}$;
\item[-] 1-morphisms between $\mathcal{C}$ and $\mathcal{D}$  are pairs $(\mathcal{A}, q)$, with $\mathcal{A}$ a fusion category, and $q$ a braided functor 
$\mathcal{C}^{\mathrm{op}}\boxtimes{\mathcal{D}}\to{\mathcal{Z}(\mathcal{A})}$, where $\mathcal{Z}(\mathcal{A})$ is the Drinfel'd centre of $\mathcal{A}$;
\item[-] 2-morphisms are $\mathcal{A}$-$\mathcal{B}$ bimodules $M$;
\item[-] 3-morphisms are bimodule functors;
\item[-] 4-morphisms are bimodule natural transformations;
\end{itemize}
Recently \cite{freed4}, the invertible objects in ${\rm Braid}^{\otimes}$ have been investigated: they are exactly the modular tensor categories. They are also fully dualizable. Let then $\mathcal{C}$ be a modular tensor category, and consider the invertible fully extended 4-dimensional TQFT $Z$ induced by $\mathcal{C}$. Also, let $(\mathcal{A}, q)$ be a 1-morphism from $\mathrm{Vect}$ (i.e., from the monoidal unit of ${\rm Braid}^{\otimes}$) to
$\mathcal{C}$, i.e., let $q$ be a braided functor $\mathcal{C}\cong \mathrm{Vect}^{\mathrm{op}}\boxtimes{\mathcal{C}} \to \mathcal{Z}(\mathcal{A})$ for some fusion category $\mathcal{A}$. By the results above\footnote{In the main body of the paper we have been considering only $n$-Vect as a target for a TQFT. The constructions presented there generalise to an arbitrary symmetric monoidal $(\infty,n)$-category  with duals $\mathcal{C}$ as a target, see \cite{lurie}. More precisely, when $\mathcal{C}$ takes the role of $n$-Vect, then $\Omega\mathcal{C}$ takes the role of $(n-1)$-Vect, and so on, down to $\Omega^n\mathcal{C}$ taking the role of the base field $\mathbb{K}$. In particular, it is meaningful to have the symmetric monoidal 4-category ${\rm Braid}^{\otimes}$ as a target, as we are doing here.}, to $(\mathcal{A}, q)$ there corresponds a boundary condition $\tilde{Z}$ of $Z$, and consequently a fully extended 3-dimensional anomalous theory with respect to $W^{Z}$ with values in $\Omega{\rm Braid}^{\otimes}$. We will denote with $Z^{(\mathcal{A},q)}$ this anomalous theory. Notice that if we apply the loop operator to the morphism $Z^{(\mathcal{A},q)}$ we obtain a 3-dimensional anomalous TQFT extended up to codimension 2 with values in $\Omega^{2}{\rm Braid}^{\otimes}\simeq 2\mbox{-}{\rm Vect}$.\\
On the other hand, given a modular tensor category $\mathcal{C}$, the Reshetikhin-Turaev construction also produces an anomalous 3-dimensional TQFT extended up to codimension 2, which we denote by $Z^{\rm RT}_{\mathcal{C}}$. It is very tempting then to state the following    
\begin{conjecture}\label{conject}
Let $\mathcal{C}$ be a modular tensor category. 
Then, any isomorphism $(\mathcal{A}, q)$ between $\mathrm{Vect}$ and $\mathcal{C}$ in ${\rm Braid}^{\otimes}$, i.e., any equivalence $q\colon\mathcal{C}\to{\mathcal{Z}(\mathcal{A})}$, induces a natural equivalence
\begin{equation}
Z^{\rm RT}_{\mathcal{C}}\simeq\Omega({Z^{(\mathcal{A},q)}}).
\end{equation}
%
\end{conjecture}
The conjecture above is compatible with findings in \cite{fsv1}, which studies obstructions to the existence of boundary conditions for Reshetikhin-Turaev TQFTs.
\begin{remark}
In Conjecture \ref{conject}, Reshetikhin-Turaev TQFT is regarded as an anomalous theory with respect to the 4-dimensional Crane-Yetter theory, i.e. a natural transformation of (higher) functors, rather than a functor on a central extension of ${\rm Cob}^{or}_{2}(3)$. In other words, we trade the additional structures on 1-, 2-, and 3-manifolds needed to define Reshetikhin-Turaev TQFT as functors, as for instance in \cite{turaev,walker}, with looking at them as natural transformations.
\end{remark}
\subsection{Further applications and outlook} An interesting playground to test and apply the language and results developed in this article is provided by the quantisation of classical Lagrangian field theories, as in \cite{fhlt,morton,nuiten}. In this case the TQFT is obtained via a linearisation of the (higher) stack of classical fields over $\infty$-categories of groupoid correspondences: we expect therefore the anomalous theory to retain some ``classical'' properties concerning the anomaly. A particularly amenable situation is given by (higher) Dijkgraaf-Witten theories: indeed, in this case we expect to reproduce the results obtained in \cite{fsv2} in 3-dimensions, which would provide a purely quantum field theoretic support to the ansatz therein proposed.\\
On a closely related topic, we remark that there is a version of the cobordism hypothesis to incorporate \emph{defects} between fully extended TQFTs. Indeed, a boundary condition for $Z$ as presented in this article can be regarded as a defect between the trivial theory and $Z$. One can then investigate morphisms between two arbitrary $n$-dimensional TQFTs of moduli level $m$, with $m>0$: we expect the structure involved in this case to be richer than the case $m=0$,  where the $(\infty, n-1)$-category of morphisms forms a groupoid.  
\end{document}